\documentclass[11pt]{amsart}
\usepackage{amssymb,latexsym,amsmath}
\usepackage[mathscr]{eucal}
\usepackage{tikz-cd}
\usepackage{graphicx}

 \usepackage{epstopdf}
  \usepackage{color} 
 \usepackage{amsthm,graphpap}

 \newcommand\s{{\mathrm{s}}}

\long\def\symbolfootnote[#1]#2{\begingroup%
\def\thefootnote{\fnsymbol{footnote}}\footnote[#1]{#2}\endgroup}

\newcommand\bsm{ \begin{smallmatrix}}
\newcommand\esm{\end{smallmatrix} }
\newcommand\bbm{\left[\begin{smallmatrix}}
\newcommand\ebm{\end{smallmatrix}\right]}
\newcommand\bcs{\begin{cases}}
\newcommand\ecs{\end{cases}}

  \renewcommand\arraystretch{1.25}

 \newcommand\fg{\mathfrak{g}}

\newcommand{\R}{\mathbb{R}}

\newcommand{\cA}{\mathscr{A}}

\newcommand{\cF}{{\mathscr{F}}}

\newcommand{\cH}{{\mathscr{H}}}

\newcommand{\cM}{{\mathscr{M}}}

\newcommand{\cO}{{\mathscr{O}}}

\newcommand{\cZ}{{\mathscr{Z}}}

\newcommand\bpm{\begin{pmatrix}}
\newcommand\epm{\end{pmatrix}}

 \renewcommand{\part}{\partial}

\newcommand\Ga{{\Gamma}}

\newcommand\tPhi{{\widetilde{\Phi}}}

\newcommand\bsl{\backslash}

\newcounter{demo}[equation]

 \newtheoremstyle{mytheo}% name
 {3pt}%      Space above
  {3pt}%      Space below
  {\itshape}%         Body font
  {}%         Indent amount (empty = no indent, \parindent = para indent)
  {\scshape}% Thm head font
  {:}%        Punctuation after thm head
  {.5em}%     Space after thm head: " " = normal interword space;
  {}%         Thm head spec (can be left empty, meaning `normal')

\theoremstyle{mytheo}

\newtheoremstyle{note}% name
  {3pt}%      Space above
  {3pt}%      Space below
  {}%         Body font
  {}%         Indent amount (empty = no indent, \parindent = para indent)
  {\bfseries}% Thm head font
  {:}%        Punctuation after thm head
  {.5em}%     Space after thm head: " " = normal interword space;
  {}%         Thm head spec (can be left empty, meaning `normal')
\theoremstyle{note}

\theoremstyle{remark}

\newcommand\bC{\mathbf{C}}
\newcommand\bP{\mathbf{P}}

\newcommand\ol[1]{\overline{#1}}

\input{xy}
\xyoption{all}

\renewcommand{\thefootnote}{\arabic{footnote}}
\numberwithin{equation}{subsection}

\def\Schmid{Sc}

\theoremstyle{plain}
\newtheorem{theorem}[equation]{Theorem}
\newtheorem{proposition}[equation]{Proposition}
\newtheorem{corollary}[equation]{Corollary}
\newtheorem{conjecture}[equation]{Conjecture}
\newtheorem{lemma}[equation]{Lemma}

\newtheorem*{claim*}{Claim}

\theoremstyle{definition}
\newtheorem{definition}[equation]{Definition}

\newtheorem*{example*}{Example}
\theoremstyle{remark}
\newtheorem{remark}[equation]{Remark}

\def\half{\tfrac{1}{2}}
\def\sur{\twoheadrightarrow}
\def\inj{\hookrightarrow}

\newcommand{\mystack}[2]{\ensuremath{ \substack{ \hbox{\tiny{${#1}$}} \\ \hbox{\tiny{${#2}$}} }} }
\def\tinyb{{\hbox{\tiny{$\bullet$}}}}

\def\cksto{\hat{\longrightarrow}}

\def\tand{\quad\hbox{and}\quad}
\def\sA{\mathscr{A}}
\def\tAd{\mathrm{Ad}}
\def\tad{\mathrm{ad}}
\def\tAut{\mathrm{Aut}}

\def\a{\alpha}
\def\olB{\overline{B}}

\def\b{\beta}
\def\bC{\mathbb{C}}

\def\d{\delta}

\def\td{\mathrm{d}}
\def\tdet{\mathrm{det}}

\def\tdim{\mathrm{dim}}

\def\tEnd{\mathrm{End}}

\def\e{\varepsilon}

\def\cF{\mathcal{F}}
\def\cFe{\mathcal{F}_\mathrm{e}}
\def\sF{\mathscr{F}}

\def\GsD{\Gamma\backslash D}

\def\tGL{\mathrm{GL}}
\def\fg{\mathfrak{g}}
\def\tGr{\mathrm{Gr}}

\def\sH{\mathscr{H}}

\def\cH{\mathcal{H}}

\def\fh{\mathfrak{h}}
\def\bi{\hbox{\footnotesize$\sqrt{-1}$}}

\def\tker{\mathrm{ker}}

\def\sL{\mathscr{L}}

\def\Le{\Lambda_\mathrm{e}}

\def\olS{\hat\wp{}^0}
\def\olP{\overline{\wp}{}^0}

\def\bN{\mathbf{N}}
\def\sN{\mathscr{N}}
\def\fn{\mathfrak{n}}

\def\sO{\mathscr{O}}
\def\cO{\mathcal{O}}
\def\op{\oplus}
\def\ot{\otimes}

\def\fP{\mathfrak{P}}
\def\fp{\mathfrak{p}}
\def\tPhi{\widetilde\Phi}
\def\Phie{\Phi^0}
\def\hPhie{\hat\Phi{}^0}
\def\bR{\mathbb{R}}

\def\fsl{\mathfrak{sl}}
\def\tSL{\mathrm{SL}}

\def\s{\sigma}
\def\tStab{\mathrm{Stab}}
\def\tspan{\mathrm{span}}

\def\sU{\mathscr{U}}
\def\olU{\overline{\sU}}

\def\sV{\mathscr{V}}
\def\w{\omega}

\def\sfv{\mathsf{v}}

\def\bY{\mathbf{Y}}
\def\sY{\mathscr{Y}}
\def\z{\zeta}
\def\cZ{\mathcal{Z}}
\def\fz{\mathfrak{z}}
\def\bZ{\mathbb{Z}}

\def\sb{{\hbox{\tiny{$\bullet$}}}}

\newenvironment{sblist}{ 
  \begin{list}{$\sb$}
   {\usecounter{cnt} \setlength{\itemsep}{2pt}
    \setlength{\leftmargin}{20pt} \setlength{\labelwidth}{20pt}
    \setlength{\listparindent}{20pt} }
   }
   {\end{list}}

\usepackage{enumitem}
\newenvironment{i_list}
  {\begin{enumerate}[label=(\roman*),itemsep=3pt,leftmargin=25pt,listparindent=20pt]}
  {\end{enumerate}}
\newenvironment{i_list_emph}
  {\begin{enumerate}[label=\emph{(\roman*)},itemsep=3pt,leftmargin=25pt,listparindent=20pt]}
  {\end{enumerate}}

\newcommand{\hsp}[1]{{\hbox{\hspace{#1}}}}
\begin{document}
%------------------------------

\title[Period mappings and the augmented Hodge line bundle]
  {Period mappings and properties of the augmented Hodge line bundle} 

\author[Green]{Mark Green}
\email{mlg@math.ucla.edu}
\address{UCLA Mathematics Department, Box 951555, Los Angeles, CA 90095-1555}

\author[Griffiths]{Phillip Griffiths}
\email{pg@math.ias.edu}
\address{Institute for Advanced Study, 1 Einstein Drive, Princeton, NJ 08540}
\address{University of Miami, Department of Mathematics, 1365 Memorial Drive, Ungar 515, Coral Gables, FL  33146}

\author[Laza]{Radu Laza}
\email{radu.laza@stonybrook.edu}
\address{Mathematics Department, Stony Brook University, Stony Brook, NY 11794-3651}
\thanks{Laza was partially supported by NSF grants DMS 1802128 and 1361143.}

\author[Robles]{Colleen Robles}
\email{robles@math.duke.edu}
\address{Mathematics Department, Duke University, Box 90320, Durham, NC  27708-0320}
\thanks{Robles was partially supported from the NSF via grants DMS 1611939 and 1361120.}

\begin{abstract}
Let $\wp\subseteq \Ga\bsl D$ be the image of a period map.  We discuss progress towards a conjectural Hodge theoretic completion $\olP$, an analogue of the Satake-Baily-Borel compactification in the classical case.  The set $\olP$ is defined and given the structure of a compact Hausdorff topological space.  We conjecture that it admits the structure of a compact complex analytic variety.  We verify this conjecture when $\tdim\,\wp \le 2$.  In general, $\olP$ admits a finite cover $\olS$ (also a compact Hausdorff space, and constructed from Stein factorizations of period maps).  Assuming that $\olS$ is a compact complex analytic variety, we show that a lift of the augmented Hodge line bundle $\Lambda$ extends to an ample line bundle, giving  $\olS$ the structure of a projective normal variety. Our arguments rely on refined positivity properties of Chern forms associated to various Hodge bundles; properties that might be of independent interest.
\end{abstract}
\date{July 2021}
\maketitle
\setcounter{tocdepth}{1}
\tableofcontents

%------------------------------

% SECTIONS OF THE PAPER
%------------------------------
\section{Introduction}  \label{S:intro}
%------------------------------

%------------------------------
%\subsection{Overview} \label{S:mot}
%------------------------------

Let $\cM$ be the moduli space for smooth varieties $X$ of general type and with given numerical characters.  In a sweeping generalization of the Deligne--Mumford compactification $\overline{\cM}_g$ of the moduli space of curves, Koll\'ar, Shepherd-Barron and Alexeev (KSBA), with contributions of many others, have constructed a canonical projective completion $\ol\cM$ with geometric meaning (see \cite{KolSur} and the references therein).  
 However, even in the case of surfaces of general type with small invariants, little is known  towards a classification of the boundary varieties, and about the global structure of the moduli space and its boundary $\part\cM$.  
A basic idea is to study $\cM$ and its compactifications by using a natural invariant, namely a period mapping $\Phi:\cM\to\Ga\bsl D$, which associates to a smooth variety its relevant cohomology group endowed with a Hodge structure.  
Since $\Ga\bsl D$ has a rich structure,  the period map should be a powerful tool for understanding the structure of $\cM$ and $\ol\cM$. 

This is indeed the case for  the study of moduli spaces of curves, abelian varieties, $K3$ surfaces, and a few other related cases. Beyond these classical cases,  to our knowledge, little is known in terms of the behavior of period maps for compactified moduli spaces. Among the simplest non-classical cases are the surfaces of general type with small invariants, for instance the case of H-surfaces and I-surfaces (surfaces of general type with $p_g=2$, $q=0$, $K^2=2$ or $1$ respectively). For such surfaces, one has a reasonable hold on the geometry of the KSBA degenerations (e.g. see  \cite{FPR17}). In current work in progress, we have obtained a number of partial results about an extended period map for H and I-surfaces. The emerging picture (e.g. see \cite{PGgaf3}) from this investigation is that this extended period map can be a very effective way to understand and give structure to the compactified moduli space $\ol\cM$.

In this paper, we focus on a piece of  this program. Namely, given the image of a period map $\wp :=  \Phi(\cM)\subset \GsD$, we are interested in a Hodge theoretic completion $\olP$ of it. Before going into details, we recall that  the subject naturally splits into a \emph{classical case}, when $D$ is a Hermitian symmetric domain, and a complementary \emph{non-classical case}. The classical case is well understood: there is a canonical, but quite singular, Satake-Baily-Borel projective compactification of $\GsD$ (\cite{Sa}, \cite{BB}), which admits various (partial) toroidal resolutions (\cite{AMRT}). 
In contrast, much less is known in the non-classical case (see however \cite{KU,KNU}, \cite{BBT18}, and \cite{GGRinfty} for some related work). 

%\begin{remark} \label{R:cc}
%We will repeatedly refer to \emph{the classical case}. This is the case when the period domain $D$ is Hermitian symmetric (or more generally, $D$ is an unconstrained Mumford-Tate domain), and $\Gamma$ is an arithmetic group. Geometrically, this corresponds to period maps for abelian varieties or K3-type objects (e.g. K3's, hyper-K\"ahler manifolds, cubic fourfolds). Occasionally, when we say classical case, we implicitly assume also $\wp=\Gamma\bsl D$.
%\end{remark}

\subsection{A general Baily-Borel type compactification} 
Here, we consider a period map
\begin{equation}
\label{1.1}
\Phi:B\to\Ga\bsl D\end{equation}  with $B$ a smooth, quasi-projective variety, and $D$ a \emph{Mumford--Tate domain}\footnote{Period domains are Mumford--Tate domains, but the later are more general.  It is necessary that we work in this greater generality; even if we are interested in period maps with target $D$ a period domain. Some of our arguments are inductive, and need the extra flexibility given by $D$ being a Mumford-Tate domain. Many of the results in the literature that are stated for period domains, are in fact true also for Mumford-Tate domains (see esp. \cite{KP16,KPR}).}  (e.g. see \cite{GGK}).   We  fix a smooth projective compactification $\ol B$ such that $Z=\ol B\bsl B$ is a reduced simple normal crossing divisor.  We further assume that the local monodromies $T_i = \exp(N_i)$ around the irreducible branches $Z_i$ of $Z$ are unipotent.    We denote by 
\begin{equation}\label{1.2}
  \wp \ := \ \Phi(B) \ \subset \ \GsD \,
\end{equation}
the image of such a period map. %We sometimes assume additionally that $\Phi$ is generically injective (i.e. generic Torelli holds); this is essentially the ``basic case''.

Our goal is to construct a completion $\olP$ of the image of the period map $\wp$, and an extension 
\begin{equation}\label{E:Phie}
  \Phi^0 : \olB \ \to \ \olP
\end{equation}
of the period map \eqref{1.1}.  In the classical case, the existence of $\olP$ is a consequence of the Baily-Borel theory \cite{BB}, while the extended period map \eqref{E:Phie} is due to Borel \cite{Bextn}. Furthermore, still in the classical case, both $\olP$ and  $\Phi^0$ are algebraic. In particular, $\wp$ is a quasi-projective variety, and frequently $\wp= \GsD$. In contrast, in the general non-classical case,   $\GsD$ does not carry, even abstractly, an algebraic structure (cf. \cite{GRT}). Nonetheless, when restricting to images of period maps, various algebraicity  properties are known or expected (see esp. \cite{G2}, \cite{Som}, \cite{BBT18}).  Most relevant here are the recent results of Bakker--Brunebarbe--Tsimerman \cite{BBT18} who proved  that
the image  $\wp$ of a period map is a quasi-projective variety, and that the augmented Hodge line bundle $\Lambda$ (see Definition \ref{dfn:ahlb}) is ample on $\wp$. In this context, we report here on steps towards the construction of a compactified version \eqref{E:Phie} for the period map, which furthermore we show to be (conditionally) algebraic.

%\begin{remark} \label{R:BBT}
%Most recently, a significant breakthrough has been obtained by Bakker--Brunebarbe--Tsimerman (BBT) \cite{BBT18} who proved using $o$-minimality techniques that
%, under the additional assumption that $\Gamma$ is arithmetic, 
%the image  $\wp$ of the period map is a quasi-projective variety.  Of course, this immediately gives a projective completion of $\wp$.  However, this falls short of what we want as it is not known what Hodge-theoretic information is encoded in the boundary, or whether there is an extension \eqref{E:Phie}, both of which are important for applications.  BBT show that the (augmented) Hodge line bundle $\Lambda\to B$ is semi-ample.  If one could also show that the (extended, augmented) Hodge line bundle $\Le \to \olB$ is semi-ample, then one could take $\overline\wp = \mathrm{Proj}\,R(\olB,\Le)$, with $R(\olB,\Le)$ the ring $\op_d\,H^0(\olB,\Le^{\ot d})$ of sections, and the associated $\olB \to \overline\wp$ would give the desired extension \eqref{E:Phie}.
%\end{remark}

Inspired by the situation in the classical case, and the analysis of degenerations of Hodge structure for surfaces of general type (e.g. Remark \ref{rem:ksba} below), it is clear that a compactification 
\begin{equation}\label{defP0}
  \olP \ := \ \bigcup_W \wp^0_W
\end{equation}
of the image $\wp$ of the period map should be obtained by gluing on pieces $\wp^0_W$ (see \eqref{def-stratapw}) parameterizing the possible graded quotients of the limit mixed Hodge structure (LMHS) arising from the period map $\Phi$ and the normal crossing compactification $B\subset \olB$ of the base (N.B. the index $W$ corresponds to a natural stratification $\{Z_W\}_W$ of $\olB$, with strata admitting proper period maps $\Phi^0_W : Z_W \to \wp^0_W$;\footnote{While we defer the details to the discussion in Section \ref{S:basics}, we point out that already here, one needs to work in the more general set-up of Mumford-Tate domains.} the unique open stratum accounts for $\Phi : B \to \wp$). For instance, in the classical case of abelian varieties $\olP=\mathfrak A_g^*=\mathfrak A_g\sqcup \mathfrak A_{g-1}\sqcup\dots\sqcup\mathfrak A_{0}$. The general case  is discussed in Section \ref{S:basics} below. Once this is done, one obtains a set-theoretic extension 
\begin{equation}\label{defPext}
  \Phie : \olB \ \to \ \olP \,,
\end{equation}
of $\Phi : B \to \wp$, which is naturally stratified by analytic pieces $\left.\Phie\right|_{Z_W} := \Phi^0_W$.  It is a consequence of the several-variable $\tSL(2)$--orbit theorem  \cite{CKS1} that  $\olP$ can be given the structure of a Hausdorff topological space (compatible with the existing topology on $\wp^0_W$ as an analytic variety). In fact, $\olP$ is naturally a stratified topological spaces and the structure of how the strata fit together depends on deep properties of LMHS's, including the weight filtration property. Additionally, the strata $\wp^0_W$ are quasi-projective varieties (cf. \cite{BBT18}). Thus, it is natural to expect the following:
\begin{conjecture}\label{conj:a}
The image $\olP = \Phie(\olB)$ is a projective completion of the image of the period map $\wp$. Furthermore, the extended period map $\Phie: \olB \to \olP$ is an algebraic morphism.
\end{conjecture}

\begin{remark}[\emph{The conjecture holds when $D$ is Hermitian}]\label{R:herm}
In the classical case (when $D$ is Hermitian symmetric and $\Gamma$ is arithmetic), $\Ga\bsl D$ is a quasi-projective variety, with a projective compactification $(\Ga\bsl D)^*$ (cf. \cite{Sa, BB}). Furthermore, the Borel Extension Theorem \cite{Bextn} yields an extension $\Phie:\ol B\to (\Ga\bsl D)^*$  of the period map \eqref{1.1} to an algebraic map. In this situation, we can take $\olP$ to be the closure of $\wp$ in $(\Ga\bsl D)^*$. In general, such an argument would not work: the quotient $\Ga\bsl D$ has no algebraic structure (\cite{GRT}), and meaningful compactifications are expected only in the horizontal directions.  
\end{remark}

\begin{remark}[\emph{KSBA vs. Baily-Borel compactifications for surfaces}]\label{rem:ksba}
Since the singularities of limit objects occurring in KSBA compactifications are du Bois (cf. \cite{KK}), standard arguments (e.g. see \cite{KL1,KLS}) show that for surfaces the mixed Hodge structure of a KSBA stable surface $S_0$ determines the graded pieces of any KSBA smoothing of $S_0$. Thus, at least set-theoretically for surfaces of general type there is a map $\overline{\cM}\to  \olP$ between the KSBA and Baily-Borel compactifications, analogous to the Torelli map $\overline{\cM}_g \to \mathfrak A_g^*$ for curves. It is of course, natural to speculate that (at least after passing to normalizations) $\overline{\cM}\to  \olP$ is in fact an algebraic morphism between projective varieties. This will be discussed elsewhere. On the other hand, starting with dimension $3$, the map $\overline{\cM}\to  \olP$ does not exist even set-theoretically. This is explained by the difference between the standard Hodge bundle (relevant to KSBA compactifications) and the augmented Hodge bundle (relevant to Baily-Borel type compactifications), which is visible only starting with dimension $3$. 
\end{remark}

For various technical reasons, it is in fact preferable not to work with $\Phi^0 : \olB \to \olP$ but rather a finite cover 
\[
\begin{tikzcd}
  \olB \arrow[r,"\hPhie"'] \arrow[rr,bend left=20,"\Phie"] 
  & \olS \arrow[r] & \olP\,,
\end{tikzcd}
\]
of it which arises from the Stein factorizations of the period maps $\Phi^0_W$ on the strata (N.B. by construction, $\Phi^0_W$ are all arranged to be proper; see \S\ref{S:stein}). The set $\hat\wp{}^0$ is naturally a Hausdorff topological space covered  by a stratification with normal, quasi-projective strata and, importantly, the fibres of $\hPhie$ are \emph{connected, compact} algebraic subvarieties of $\olB$. In the context of Conjecture \ref{conj:a}, it is natural to expect that $\hat\wp{}^0$ is a normal projective variety.\footnote{One of the advantages of passing to $\hat\wp{}^0$ is to gain normality.} We are not able to prove this at this time. We are making instead a weaker conjecture, namely $\hat\wp{}^0$ can be upgraded to a complex analytic variety. If this indeed the case, we are able to show that  $\hat\wp{}^0$ is algebraic. Specifically, our main conjecture in this paper is the following:

\begin{conjecture}\label{th:a0}	
The compact Hausdorff space $\olS$ admits the structure of a normal complex analytic variety with the properties that:
\begin{i_list_emph}
\item 
The extension $\hPhie : \olB \to \olS$ is an analytic map.
\item
The restriction of the analytic structure on $\olS$ to the strata $\hat\wp{}^0_W$ coincides with the natural analytic structure on $\hat\wp{}^0_W$.
\end{i_list_emph}
\end{conjecture}

\begin{remark}
The conjecture holds when $D$ is Hermitian (see Remark \ref{R:herm}).
\end{remark}

\begin{remark}[\emph{The conjecture holds when $\tdim \wp = 1$}]\label{R:dim=1}
In the case that $\wp = \Phi(B)$ is one--dimensional, the conjecture is a consequence of non-trivial results of Sommese \cite{Som73} and Cattani-Deligne-Kaplan \cite{CDK}. We note that a key ingredient in Sommese's argument, the {\it Siegel property}, fails\footnote{An explicit example of this was given in an earlier version of our paper, which is available on arXiv.} in general for higher dimensional bases. This illustrates one of the challenges of passing from the one-variable case to multiple variables.
\end{remark}

Our first result (see \S\ref{S:prfdimB=2}) is the two dimensional version of the above conjecture. 

\begin{theorem}\label{T:dimB=2}
If $\tdim\,B = 2$ and $\Phi : B \to \Gamma \backslash D$ satisfies local Torelli, Conjecture \ref{th:a0} holds.
\end{theorem}

Moving on to the general case, we introduce in \S\ref{S:eaHb} the {\it augmented Hodge line bundle} $\Lambda$ living on $B$ (aka {\it Griffiths' bundle}) and the extended augmented Hodge bundle $\Le$ living on $\overline B$. The extended Hodge bundle descends to $\hat\wp^0$.

\begin{theorem}[{\cite{GGRinfty}}] \label{T:descends}
For some $m\ge 1$, the line bundle $\Le^m$ on $\overline B$ descends to a line bundle on $\hat\wp^0$.
\end{theorem}

\begin{remark}
In the case that $\Gamma$ is neat, we may take $m=1$ in the theorem, cf.\cite{GGRinfty}.
\end{remark}

For notational simplicity, we will drop the index $m$ -- that is, we assume $m=1$ -- and work with $\Le$.  We note, however, that the arguments that follow all apply to the general case, with $\Le^{\ot m}$ in place of $\Le$.

A priori, $\Le \to \olS$ is a topological line bundle with the property that its restrictions to the analytic strata $\wp^0_W \subset \olS$ are holomorphic.  However, if Conjecture \ref{th:a0} holds, then  $\Le \to \olS$ naturally admits the structure of a holomorphic line bundle that is compatible with the existing holomorphic structure of $\left.\Le\right|_{\wp^0_W}$.
With these preliminaries, we can state our main result, the conditional (on Conjecture \ref{th:a0}) algebraicity of the compactified $\olS$.

\begin{theorem}\label{th:b}
Assume that the differential of $\Phi: B \to \Gamma \backslash D$ is generically injective, and that Conjecture \ref{th:a0} holds.  Then $\Le\to\olS$ is ample.
\end{theorem}

\begin{remark} 
The assumption that the differential of $\Phi$ is generically injective is natural, and for us the fundamental case. Namely,  we are interested in completing $\wp = \Phi(B)$.  Thus, if needed, we could replace $B$ and $\olB$ by a suitable log resolution of a compactification of $\wp$.   
\end{remark}

\begin{remark}
The algebraic structure given by Theorem \ref{th:b} is compatible with the algebraic structure (on strata) coming from \cite{BBT18}. Specifically, if one can show that $\Le$ is semi-ample on $\overline B$, then  $\overline{\wp} = \mathrm{Proj}R(\olB,\Le)$ is a projective compactification of $\wp$, and we have a commutative diagram 
\begin{equation}\label{diag-BBT}
  \begin{tikzcd}
  \olB \arrow[r,"\hPhie"'] \arrow[rrr,bend left=25,"\Phie"] 
  & \olS \arrow[r]
  & \mathrm{Proj}\,R(\olB,\Le) \arrow[r] 
  & \olP
  \end{tikzcd}
\end{equation}
relating the various spaces arising in our paper. 
\end{remark}

The proof of Theorem \ref{th:b} is in spirit analogous to the one used by Kodaira to show that over a compact, complex manifold a line bundle with positive Chern class in the differential-geometric sense is ample.  The argument depends on some rather subtle properties of the Chern form $c_1(\Lambda)$ that are discussed in \S\ref{S:chern} below. We note that the discussion of Chern forms is perhaps more general than strictly needed here. However, we choose to include it as it might be of independent interest, and in fact it was already used in some other work (e.g. \cite{phvb}). 

%------------------------------
\subsection{Properties of the Chern forms} \label{S:chern}
%------------------------------
Let $
  \tGr^p\cF = \cF^p/\cF^{p+1} \ \to \ B$
denote the quotient Hodge bundles.  The polarization induces Hermitian forms on these bundles.  Let $c_r(\tGr^p\cF) \in \cA^{r,r}(B)$ denote the Chern forms (with $\cA^{r,r}$ being the sheaf of smooth $(r,r)$--forms).   Assuming (as we may by passing to the proper extension of the period map) that all monodromy logarithms $N_i\ne 0$, these forms do not have smooth extensions to any larger open set $B \subsetneq B' \subset \olB$.  However, they do define currents on $\olB$ (cf. \cite[(5.23)]{CKS1}).  Using these currents and some subtle properties satisfied by them allows us to argue essentially as in the smooth case (i.e. when $B=\overline B$). 

We let $c_r(\tGr^p\cFe)$ denote these currents.   The notation is justified by the fact that Cattani--Kaplan--Schmid  \cite{CKS1} showed that the first Chern form/current $c_1(\tGr^p\cFe)$ computes the first Chern class of the extension $\tGr^p\cFe = \cFe^p/\cFe^{p-1}$ to $\olB$. Subsequently, Koll\'ar \cite[Theorem 5.20]{Ko2} generalized this statement to all the Chern forms.  

One point that is new here is that the currents $c_r(\tGr^p\cFe)$ may be thought of as admitting well-defined restrictions to smooth forms on the strata $Z_I^*$ of the boundary $\overline B\setminus B$.  This is made precise in Theorem \ref{th:c} below.  It is an interesting corollary that we can multiply the Chern forms (Remark \ref{R:currents}).

To explain our result, we recall that the period maps on strata (see \eqref{PhiI}) endow the  strata $Z_I^*$ with Hodge vector bundles $\cF^p_I$, and the polarizations induce Hermitian metrics on the quotient bundles
\[
  \tGr^p\cF_I = \cF_I^p/\cF_I^{p+1} \ \to \ Z_I^*\,.
\]
Let $c_r(\tGr^p\cF_I) \in \cA^{r,r}(Z_I^*)$ denote the Chern forms.\footnote{There is a choice of polarization on the strata; while the Hermitian metric depends on this choice, the Chern form does not (Remark \ref{R:dfnhI}).}
With these notations, Theorem \ref{th:c} says informally that 
the current $c_r(\tGr^p\cFe)$ has a well-defined restriction to $Z_I^*$, where it agrees with $c_r(\tGr^p\cF_I)$; we express this as
\[
  \left.c_r(\tGr^p\cFe)\right|_{Z_I^*} \ = \ c_r(\tGr^p\cF_I) \,.
\]
  
For the precise statement, reordering the indices if necessary, we assume that $I = \{1,\ldots,k\}$.  Every point $b \in Z_I^*$ admits a local coordinate chart $(t,w): \olU \to \Delta^k \times \Delta^\ell$ so that $Z_i \cap \olU = \{ t_i=0\}$ (with $\Delta \subset \bC$ the unit disc).  In particular, $Z_I^* \cap \olU = \{ t_1,\ldots,t_k = 0 \}$.   Let 
\[
  \left.c_r(\tGr^p\cF)\right|_{\sU} \ = \ 
  \sum_{\mystack{|A|+|C|=r}{|B|+|D|=r}} f^p_{\ell,A,B,C,D}\,
  \td w_A \wedge \td \bar w_B \wedge \td t_C \wedge \td \bar t_D
\]
be the local coordinate expression over $\sU = B \cap \olU$, and let
\begin{equation}\label{E:rhoell}
  \varrho_r(\tGr^p\cF) \ = \ 
  \sum_{|A|, |B|=r} f^p_{\ell,A,B}\,
  \td w_A \wedge \td \bar w_B
\end{equation}
be the part of this expression involving only the $\td w$ and $\td \bar w$'s.  Equivalently, the interior products $i_{\partial/\partial t_i} \varrho_r(\tGr^p\cF)$ and $i_{\partial/\partial\bar t_i} \varrho_r(\tGr^p\cF)$ vanish, and 
\[
   \varrho_r(\tGr^p\cF) \ \equiv \ \left.c_r(\tGr^p\cF)\right|_{\sU}
   \qquad \hbox{ mod } \ \td t_i \,,\ \td \bar t_i \,.
\]
While $\left.c_r(\tGr^p\cF)\right|_{\sU}$ does not extend to a smooth form on $\olU$, the form $\varrho_r(\tGr^p\cF)$ does.

\begin{theorem}\label{th:c}
The form $\varrho_r(\tGr^p\cF) \in \cA^{r,r}(\sU)$ extends to a smooth well-defined form $\varrho_1(\tGr^p\cFe) \in \cA^{r,r}(\olU)$, and 
\[
  \left.\varrho_r(\tGr^p\cFe)\right|_{Z_I^* \cap \olU}
  \ = \ 
  \left. c_r(\tGr^p\cF_I) \right|_{Z_I^* \cap \olU} \,.
\]
\end{theorem}

\noindent The theorem will follow from 
\begin{equation}\label{E:thc}
  \lim_{t_1,\ldots,t_k\to0} \varrho_r(\tGr^p\cF) \ = \
  \left.c_r(\tGr^p\cF_I)\right|_{Z_I^* \cap \olU} \,,
\end{equation}
which is a special case of \eqref{E:cw}.  The existence of the limit and the equality is established in \S\ref{S:IV} by a significant elaboration of the arguments in \cite[\S5]{CKS1} and \cite[\S5]{Ko2}.  Note that the definition \eqref{E:rhoell} of $\varrho_r(\tGr^p\cFe)$ does depend on the choice of coordinates $(t,w)$.  However, the right-hand side of of \eqref{E:thc} is independent of any coordinate choice.

\begin{remark} \label{R:currents}
In general it is not possible to make sense of  the wedge product of two currents.  Theorem \ref{th:c} implies that the Chern forms are an exception: we have a well-defined notion of the wedge product amongst the various $c_r(\tGr^p\cF)$ on $\olB$ that is compatible with cup product in cohomology.  
\end{remark}

Let $
  \Lambda_I \ = \
  \tdet(\mathcal{F}_I^n) \ot \tdet(\mathcal{F}_I^{n-1}) \ot \cdots 
  \ot \tdet(\mathcal{F}_I^{\lceil (n+1)/2\rceil})$,
be the (augmented) Hodge line bundle over $Z_I^*$.  Let $\varrho_1(\Lambda) \in \cA^{1,1}(\sU)$ be the obvious analog of \eqref{E:rhoell} for the Chern form $c_1(\Lambda) \in \cA^{1,1}(B)$ of the Hodge line bundle.

\begin{corollary}\label{cor:c}
The form $\varrho_1(\Lambda) \in \cA^{1,1}(\sU)$ extends to a smooth well-defined form $\varrho_r(\Le) \in \cA^{1,1}(\olU)$, and 
\[
  \left.\varrho_1(\Le)\right|_{Z_I^* \cap \olU}
  \ = \ 
  \left. c_1(\Lambda_I) \right|_{Z_I \cap \olU} \,.
\]
\end{corollary}

\noindent Informally, this means that the current $c_1(\Le)$ admits a well-defined restriction to $Z_I^*$ where it coincides with the Chern form $c_1(\Lambda_I)$.

Our primary applications of Corollary \ref{cor:c}   are  \eqref{E:L|C} and Corollary \ref{C:snef} below (see also \cite{phvb} for some further applications).  Namely,   let $C \subset \olB$ be an irreducible curve, then it follows
\begin{equation}\label{E:L|C}
   \int_C c_1(\Le) \ = \ 
   \mathrm{deg} \left( \left.\Le\right|_C \right) 
   \ \ge \ 0\,,
\end{equation}
and equality holds if and only if $\Phi^0(C)$ is a point.  Keeping the positivity \eqref{SE:c1pos} in mind, Theorem \ref{T:descends} and Corollary \ref{cor:c}, yield

\begin{corollary}\label{C:snef}
Assume that Conjecture \ref{th:a0} holds.  Then $\Le \to \olS$ is strictly nef.
\end{corollary}

%------------------------------
\section{Completion of $\wp$ and extended Hodge bundles.}  \label{S:basics}
%------------------------------
Throughout the paper, we are in the set-up considered in the introduction; the standard set-up of degeneration of Hodge structure. Namely, we fix a period map $\Phi$ as in \eqref{1.1}, and a simple normal crossing compactification $B\subset \olB$ of the base satisfying additionally that the local monodromies around the boundary divisors are unipotent. As before, let $\wp= \Phi(B)$ be the image of period map. 

In this review section, we discuss the relative (to $\olB$) completion $\olP$ and its Stein factorization cover $\hat \wp{}^0\to \olP$. It is a consequence of the theory of degenerations of Hodge structure (esp. \cite{CKS1}) that  $\olP$ (and similarly $\hat\wp{}^0$) is a stratified topological space with the strata $\wp^0_W$ being again images of period maps. The ultimate goal of our study is to show that $\olP$ has an algebraic structure and that $\Phi$ extends to an algebraic morphism  $\Phie$ \eqref{defPext}. In this direction, we review the augmented Hodge bundle $\Lambda$ on $B$ and its extension $\Le$ on $\olB$. It was known since the seventies, that $\Lambda$ satisfies strong positivity properties on $\wp$. More recently,  Bakker--Brunebarbe--Tsimerman \cite{BBT18} proved that $\wp$ is indeed quasi-projective (using $o$-minimality theory) and (that the descent of) $\Lambda$ is ample on $\wp$ (using positivity arguments). Similarly, the strata $\wp^0_W$  (and $\hat\wp^0_W$) occuring in the construction of $\olP$ (and $\hat\wp{}^0$ respectively) are themselves quasi-projective (cf.~discussion following \eqref{SE:c1pos}). Our goal is to glue all these pieces together into an projective variety $\olP$ (cf. Conjecture \ref{conj:a}). A first step in this direction is to 
 show that (a power of) $\Le$ descends to a line bundle on $\hat\wp{}^0$. The descent property is a consequence of the analysis undertaken  in \cite{GGRinfty} of the asymptotic behavior of the period map $\Phi$ in a neighborhood of a $\hPhie$--fibre.  In that work it is essential that those fibres are both compact and connected. In particular, the results there do not yield descent to $\olP$.

This section is complemented by Section \ref{S:rvw} of the Appendix, where we collect the  more technical aspects of  the asymptotics of the period maps. Those are particularly relevant later on when we discuss the properties of the Chern forms.

%------------------------------
\subsection{Stratified, set theoretic completion of $\Phi$} 
%------------------------------
With notations as above, we let $Z=\olB\setminus B$ be the boundary divisor, and we denote by $Z_i$ the irreducible components of $Z$, and define 
\[
  Z_I \ := \ \bigcap_{i\in I}Z_i.
\]
This in turn defines a stratification of $\ol B$ by setting 
\[
  Z^\ast_I \ := \ Z_I \backslash \left( \cup_{|J|>|I|}\, Z_J \right),
\]
the open strata obtained by removing from $Z_I$ the lower dimensional sub-strata.  By convention, the top stratum $Z_{\emptyset}^*$ corresponds to $B$, i.e. $Z_{\emptyset}^*=B$. 

It is a consequence of the nilpotent orbit theorem \cite{Sc} that over each open stratum $Z^\ast_I$ we have a variation of nilpotent orbits.  Passing to the ``associated weight-graded quotient'' we obtain period mappings
\begin{equation}\label{PhiI}
   \Phi^0_I:Z^\ast_I\to\Ga_I\bsl D^0_I \,,
\end{equation}
(see \S\ref{S:PhiI} for further details). In \cite{PoI3} it is proved that period maps can be extended to proper holomorphic maps; simply, we can assume that $B\subset \olB$ is extended to contain all the strata $Z_{\{i\}}^*$ with trivial monodromy.
The same idea can be used to patch together the period maps $\Phi^0_I$ on strata to obtain proper holomorphic maps. Namely, along each strata $Z_I^*$ there is a well-defined $\Gamma$--congruence class of weight filtrations $[W^I]$.  If we define
\[
  Z_W \ = \ \bigcup_{[W^I] = [W]} Z_I^* \,,
\]
then the maps \eqref{PhiI} with $Z_I^* \subset Z_W$ define a proper holomorphic map
\[
  \Phi^0_W : Z_W \ \to \Gamma_W \bsl D^0_W \,.
\]
On any given substrata $Z_I^* \subset Z_W$, we have a commutative diagram
\[
\begin{tikzcd}
  Z_I^* \arrow[r,"\Phi^0_I"] \arrow[d,hook] & \Ga_I\bsl D^0_I \arrow[d] \\
  Z_W \arrow[r,"\Phi^0_W"] & \Ga_W\bsl D^0_W \,,
\end{tikzcd}
\]
and the fibres of $\Phi^0_I(Z_I^*) \sur \Phi^0_W(Z_I^*)$ are finite.
The properness of $\Phi^0_W$ implies that
\begin{equation}\label{def-stratapw}
  \wp^0_W \ := \ \Phi_W^0(Z_W) \ \subset \ \Gamma_W \backslash D_W
\end{equation}
is a complex analytic variety.  We refer to \cite[\S2.3]{GGRinfty} for further details and discussion. 

Glueing everything together, we obtain 
\[
  \olP \ := \ \bigcup_W \wp^0_W
\]
and a  set-theoretical extension 
\[
  \Phie : \olB \ \to \ \olP \,,
\]
of the period map $\Phi : B \to \wp$  by seting $\left.\Phie\right|_{Z_W} := \Phi^0_W$.  By construction, when restricting to the strata $Z_W$, we are in the analytic category and the morphisms  are proper holomorphic.  See \cite[\S2.4]{GGRinfty} for further details and discussion. 

%------------------------------
\subsection{The Stein factorization $\olS\to\olP$} \label{S:stein}
%------------------------------
As already noted, it is important in our arguments to work with period maps with connected fibers. As usual, this can achieved by passing to the Stein factorization (e.g. see \cite{GR84})
\begin{equation}\label{E:hatPhi}
\begin{tikzcd}
  B \arrow[r,"\hat\Phi"'] \arrow[rr,bend left=20,"\Phi"]
  & \hat\wp \arrow[r] & \wp
\end{tikzcd}
\end{equation}
of $\Phi$, which was previously arranged to be proper.  Here $\hat\wp = \hat\Phi(B)$ is a normal complex analytic variety, the fibres of $\hat\Phi : B \to \hat\wp$ are connected, while those of $\hat\wp \to \wp$ are finite.  Similarly, for each $Z_W$ stratum, we consider the corresponding  Stein factorizations
\begin{equation}\label{E:steinW}
\begin{tikzcd}
  Z_W \arrow[r,"\hat\Phi{}^0_W"'] \arrow[rr,bend left=20,"\Phi^0_W"]
  & \hat\wp{}^0_W \arrow[r] & \wp^0_W
\end{tikzcd}
\end{equation}
of $\Phi^0_W$. Again, the $\hat\wp{}^0_W$ are normal complex analytic varieties, the fibres of $\hat\Phi{}^0_W : Z_W \to \hat\wp{}^0_W$ are connected, and the fibres of $\hat\wp{}^0_W \sur \wp^0_W$ are finite.  Taking the union
\[
  \olS \ := \ \bigcup_W \hat\wp{}^0_W \tand 
\]
we obtain set-theoretically the extension
\[
  \hPhie : \olB \ \to \ \olS \,,
\]
of the map $\hat\Phi: B \to \hat\wp$, also defined strata-wise by \eqref{E:steinW}.   This yields a ``Stein factorization''
\begin{equation}\label{E:stein}
\begin{tikzcd}
  \olB \arrow[r,"\hPhie"'] \arrow[rr,bend left=20,"\Phie"] 
  & \olS \arrow[r] & \olP\,,
\end{tikzcd}
\end{equation}
of the extended period map $\Phie:\olB\to \olP$.  By construction the map $\olS \sur \olP$ is finite, and the restriction to strata are honest Stein factorizations.   See \cite[\S2.5]{GGRinfty} for further details and discussion. 

%------------------------------
\subsection{Topological completions} 
%------------------------------
The asymptotic behavior of the period maps guarantees the fact both $\olP$ and $\olS$ admit a compact Hausdorff topology with the following properties:
\begin{itemize}
\item[(i)] $\wp \subset \olP$ and $\hat\wp \subset \olS$ are dense open subsets; 
\item[(ii)] the induced subspace topology on $\wp^0_W$ and $\hat\wp{}^0_W$ coincides with the topology on these spaces as complex analytic varieties;
\item[(iii)] the maps in \eqref{E:stein} are all continuous and proper.  
\end{itemize}
The assertions (i) and (ii) above are established for $\olP$ in \cite[Proposition 2.25]{GGRinfty}; the argument there applies in a straightforward manner to $\olS$; and (iii) is immediate.

The topological space $\olP$ is our proposed SBB type compactification for general images of period maps. We expect $\olP$ to be an analytic space, and even algebraic (Conjecture \ref{conj:a}). Here, we only note that the cover $\olS$ is more amenable to analysis. For instance, assuming $\olS$ has an analytic structure (Conjecture \ref{th:a0}), then $\olS$ is automatically normal, while of course $\olP$ (and even $\wp$) does not need to be so. Our general strategy is to study $\olS$ and then use descent techniques (such as those of \cite{kollar_quot}) to get a hold on $\olP$.

%We emphasize that this is a ``relative'' construction: given the data $(\olB, Z; \Phi)$ a $\olP$ is produced.  (We do \emph{not} construct any sort of compactification or ``horizontal completion'' of $\Gamma\bsl D$.)  % There will be functoriality properties of this construction (not discussed here), and possibly one might formalize the construction in categorical language. 

%--------------------------------------
\subsection{Hodge vector bundles} \label{S:eaHb}
%--------------------------------------
The algebraicity of period maps is a consequence of strong positivity properties of Hodge bundles. Here, we review the basic set-up as relevant to us. To start, let 
\[
  \cF^n \,\subset\, \cF^{n-1} \,\subset \cdots \subset\, 
  \cF^1 \,\subset\, \cF^0
\]
denote the Hodge vector bundles over $B$.  Under the assumption that the local monodromies around the branches $Z_i$ of $Z$ are unipotent with logarithms $N_i$, it is well known (\cite{CKS1} and \cite{PS}) that there are canonical extensions 
\[
  \cFe^n \,\subset\, \cFe^{n-1} \,\subset \cdots \subset\, 
  \cFe^1 \,\subset\, \cFe^0
\]
of the Hodge vector bundles to $\olB$.  
%The \ipr\  becomes $\nabla \cF^p_\mathrm{e} \subseteq \Om^1_{\ol B}(\log Z)\otimes \cF^{p-1}_\mathrm{e}$, and $\Res_{Z_i}\nabla = N_i$ (up to a factor of $2\pi\sqrt{-1}$).

\begin{definition}\label{dfn:ahlb}
The (\emph{augmented}) \emph{Hodge line bundle}\footnote{For consistency with the existing literature, we point out that in \cite{BBT18} the Hodge line bundle is called {\it Griffiths bundle}.} (HLB) over $B$ is 
\begin{equation}\label{E:ahlb}
  \Lambda \ = \
  \tdet(\mathcal{F}^n) \ot \tdet(\mathcal{F}^{n-1}) \ot \cdots 
  \ot \tdet(\mathcal{F}^{\lceil (n+1)/2\rceil}) \,,
\end{equation}
and its extension to $\olB$ is denoted 
\begin{equation}
  \Le \ = \
  \tdet(\cFe^n) \ot \tdet(\cFe^{n-1}) \ot \cdots 
  \ot \tdet(\cFe^{\lceil (n+1)/2\rceil}) \,.
\end{equation}
\end{definition}
\begin{remark}
The Hodge bundle (i.e.~ $\tdet(\mathcal{F}^n)$) is more familiar in moduli theory. For VHS of weight $1$ or $2$, there is no difference between the augmented Hodge bundle and the Hodge bundle (compare Remark \ref{rem:ksba}). %Similarly, for classical period domains, one can restrict to considering only the standard Hodge bundle. 
\end{remark}

The key standard fact about the augmented Hodge bundle is that the Chern form $c_1(\Lambda)$ is semi-positive. In fact, given $\xi \in TB$ we have 
\begin{subequations}\label{SE:c1pos}
\begin{equation}
  c_1(\Lambda)(\xi,\overline\xi) \ = \ \| \Phi_\ast(\xi)\|^2 \,,
\end{equation}
cf.~\cite[(7.15)]{PoI3}.  Note that \emph{$c_1(\Lambda)$ is positive on on $B$ if and only if $\Phi_*$ is everywhere injective}.  Similarly, for each boundary stratum, we have 
\begin{equation}
c_1(\Lambda_I)(\xi,\overline\xi) \ = \ \| (\Phi^0_I)_\ast(\xi)\|^2 \,,
\end{equation}  
\end{subequations}
for all $\xi \in TZ_I^*$.  

These positivity properties of $\Lambda$ (and $\Le$) have strong consequences. In particular, Bakker--Brunebarbe--Tsimerman \cite[\S5]{BBT18} proved that $\Lambda$ descends to an ample line bundle on the image of the period map $\wp$. Similarly, the restrictions of $\Le$ to the strata $Z_W$ will descend to ample line bundles on the strata $\wp^0_W$. Now, if we would know that  $\Le$ is semi-ample, then we could define a closure of the image of the period map by $\overline{\wp}:= \mathrm{Proj}R(\olB,\Le)$, and this space would fit in the diagram \eqref{diag-BBT} mentioned in the introduction. For now, we only note that the extended Hodge bundle $\Le$ descends\footnote{We expect to have an appropriate descent also to $\olP$, however our argument depends on the connectedness of the fibers of the period map, which is true only for $\olS$.} to a (topological) line bundle on $\olS$. This is a consequence of \cite{GGRinfty} which analyzes the behavior of period maps in a neighborhood of a fiber. 

\begin{theorem}\label{T:descends0}
The line bundles $\tdet(\cFe^p)^{m_p} \to \olB$ are trivial on the connected fibres of $\hPhie$, for some $m_p \ge 1$, and descend to lines bundle on $\olS$.  
\end{theorem}

\begin{proof}
Let $A \subset \olB$ be a fibre of $\hPhie$.  The theorem is a corollary of \cite[Theorem 3.12]{GGRinfty} which asserts that the point $\hPhie(A) \in \olS$ admits a neighborhood $\hat\sO \subset \olS$ with the property that $\tdet(\cFe^p)^{m_p}$ is trivial over $\overline{\sO} = (\hPhie)^{-1}(\hat\sO) \subset \olB$, for some $m_p = m_p(\hat\sO) \ge 1$.  Since $\olS$ is compact, we may choose $m_p(\hat\sO)$ independent of the neighborhood.  
\end{proof}

\noindent Theorem \ref{T:descends0} immediately yields

\begin{corollary}\label{C:descends}
The line bundle $\Le^m \to \olB$ is trivial on the connected fibres of $\hPhie$, for some $m \ge 1$, and descends to a line bundle on $\olS$.
\end{corollary}

\begin{remark}
If $\Gamma$ is neat, we may take $m, m_p = 1$.  We will assume this is the case from this point on to reduce notational  clutter.  The reader who does not which to make this assumption may simply replace all instances of $\Le \to \olS$ with $\Le^m \to \olS$ without compromising the arguments that follow.
\end{remark}

%------------------------------
\section{Proofs of the algebraicity results} \label{S:3pfs}
%------------------------------
The technical core of the paper is Section \ref{S:IV}, where new properties of the Chern forms (applying to the compact $\olB$) are established. Here, assuming those results (as stated in \S\ref{S:chern} of the introduction), we prove the main Theorem \ref{th:b}. As a warm-up, we discuss the two dimensional case (Theorem \ref{T:dimB=2}), which follows by elementary consideration (from the positivity of Hodge bundle), illustrating nonetheless  some of the special properties satisfied by images of period maps. 

%------------------------------
\subsection{The case that $B$ is a surface} \label{S:prfdimB=2}
%------------------------------
In this subsection, we assume $\tdim\,B = 2$ and that $\Phi_*$ is one-to-one on $B$.

Since $\tdim\,B = 2$, the boundary divisors $Z_i$ are smooth, irreducible curves meeting transversally.  The extended period map $\Phie$ (cf. \eqref{defPext}) maps $Z_i$ to either a point or a curve.  Let
\[
  Z' \ := \ \sum_{\Phie(Z_i) = \mathrm{pt}} Z_i
  \ = \ \sum_{i=1}^m Z_i
\]
be the union of those $Z_i$ that are mapped to a point.

\begin{lemma}\label{T:(5)}
The intersection matrix $\|Z_i\cdot Z_j\|_{i,j=1}^m$ is negative definite.
\end{lemma}

\begin{proof}
By Theorem \ref{th:c}, the current $c_1(\Le)$ represents the Chern class of the augmented Hodge line bundle $\Le \to \olB$.  It follows from \eqref{SE:c1pos} that 
\begin{sblist}
\item $c_1(\Le) \ge 0$, and 
\item $\left.c_1(\Le)\right|_{Z_i} = 0$, for $i=1,\ldots,m$, so that $\Le \cdot Z_i = 0$.
\end{sblist}
Additionally, the assumption that $\Phi_*$ is one-to-one implies that $c_1(\Le)^2 > 0$.    It follows that $\Le$ lies in the positive cone in $\mathrm{Pic}(\olB)$.  We now infer the lemma from the Hodge index theorem.
\end{proof}

\begin{proof}[Proof of Theorem \ref{T:dimB=2}] Given the lemma, a result of Grauert \cite{Gra62} asserts that $Z'$ may be contracted to normal singular points on a complex analytic space $Y$. 
\end{proof}

\begin{remark}
In the classical case that the Mumford--Tate domain is Hermitian symmetric, the SBB compactification $\GsD^*$ of $\GsD$ is a normal projective variety.  Borel's extension theorem yields the morphism $\Phie : \olB \to (\GsD)^*$.  By hypothesis this morphism contracts $Z'$ to a set of points.  Lemma \ref{T:(5)} then follows from a result of Mumford \cite{Mum61}.
\end{remark}

%------------------------------
\subsection{Sufficient conditions for semiample} \label{S:prfw!=0}
%------------------------------

\begin{theorem}\label{T:w!=0}
Suppose that $c_1(\Lambda)$ is positive on $B$, and that $0 \le -K_{\olB} \cdot C$ whenever $\Phie$ collapses the curve $C \subset \olB$ to a point.  If the effective cone of 1-cycles $\mathrm{Eff}^1(\olB)$ is finitely generated,\footnote{We anticipate that this hypothesis can be removed.} then $\Le \to \olB$ is semiample.
\end{theorem}

\begin{proof}  Suppose that $c_1(\Le)$ is positive on on $B$ (equivalently, $\Phi_*$ is everywhere injective by \eqref{SE:c1pos}), and that $0 \le -K_{\olB} \cdot C$ whenever $\Phie(C)$ is a point.  In order to apply the base point free theorem \cite{KM} to show that $\Le \to \olB$ is free we need to show that
\begin{i_list}
\item \label{i:nef}
$m \Le - K_{\olB}$ is nef for $m\gg0$, and 
\item \label{i:big}
$m \Le - K_{\olB}$ is big for $m\gg0$.
\end{i_list}

We begin with nef.  If $\Phie(C)$ is a point, then $(m \Le - K_{\olB}) \cdot C = - K_{\olB} \cdot C \ge 0$ by assumption. Suppose that $\Phie(C)$ is a curve.  Then \eqref{E:L|C} and Theorem \ref{th:c} imply $(m \Le - K_{\olB}) \cdot C > 0$ for $m \ge m_o(C)$.  The hypothesis that $\mathrm{Eff}^1(\olB)$ is finitely generated allows to choose $m_o \ge m_o(C)$ for all curves $C$ (that are not contracted by $\Phi^0$).  This establishes \ref{i:nef}.

To prove bigness, Theorem \ref{T:big} (Demailly--Siu solution of the Grauert--Riemenschneider Conjecture) asserts that it suffices to show that at some point $b \in B$ we have $c_1(m \Le - K_{\olB})^d > 0$ for $m\gg0$ and with $d=\mathrm{dim}\,\olB$.  This follows from our assumption on the positivity of $c_1(\Le)$.
\end{proof}

\begin{remark}
In \cite{GGRlb} it is conjectured that under suitable local Torelli-type assumptions there exist integers $0 \le a_i \in \bZ$ and $m_0$ so that $m\Le - \sum a_i[Z_i]$ is ample for $m \ge m_0$.  And it is shown that the conjecture holds in two cases: when $Z = Z_1$ is irreducible; and when $\tdim\,B=2$.
\end{remark}

\begin{theorem}[Demailly {\cite{Dem85}}, Siu {\cite{Siu84, Siu85}}] \label{T:big}
Let $L \to X$ be a line bundle on a compact complex manifold of dimension $d$.  Suppose $L$ admits a smooth Hermitian metric with the property that the Chern form $\w$ is non-negative ($\w\ge0$) everywhere, and there exists a point $x \in X$ at which $\w^d > 0$.  Then $L$ is big.
\end{theorem}

%------------------------------
\subsection{Proof of ampleness of extended augmented Hodge line bundle} \label{S:V}
%------------------------------

The goal of this section is to prove Theorem \ref{th:b} provided that $\hat\wp{}^0$ admits the analytic structure of Conjecture \ref{th:a0}. Specifically, we show  that the augmented Hodge line bundle $\Le \to \hat\wp{}^0$ of Corollary \ref{C:descends} is ample.   With some important differences to be explained below, the argument will follow the overall strategy, and use a number of the technical points, of the proof of \cite[Theorem 1.1]{BBT18}. Our new ingredient is the properties of the Chern form $c_1(\Le)$ that are outlined in \S\ref{S:chern}. 

%------------------------------
\subsubsection{Ample line bundles} \label{S:ample}
%------------------------------
We briefly review the relevant notions of positivity for our situation. Throughout, we work with line bundles $L\to X$ over (possibly non-reduced) compact analytic varieties.
\begin{definition}
The line bundle $L \to X$ is:
\begin{sblist}
\item
\emph{ample} if for every cohorent sheaf $\sF \to X$ there exists $m_o(\sF)$ so that $H^q(X,\sF\ot L^m) = 0$ for all $q>0$ and $m \ge m_o(\sF)$;
\item
\emph{semi-ample} if there exists $0 < m \in \bZ$ so that the restriction $H^0(X,L^m) \to L_x^m$ is surjective for all $x \in X$;
\item 
\emph{big} if $h^0(X,L^m) = C m^d + \cdots$ for some $C > 0$ and where $d = \tdim\,X$ (here we assume $X$ is irreducible);
\item 
\emph{strictly nef} if for any curve $C \subset X$ we have $\mathrm{deg}( \left.L\right|_{C_\mathrm{red}} ) > 0$.
\end{sblist}
\end{definition}

We make the following observation.

\begin{lemma}\label{L:ample}
If $L \to X$ is semi-ample and strictly nef, then $L \to X$ is ample.
\end{lemma}

\begin{proof}
For $m\gg0$ the complete linear system $|mL|$ induces a morphism $\psi : X \to \bP^N$ so that $\psi^*(\cO_{\bP^N}(1)) = L^m$.  It follows that no curve $C \subset X$ is contracted to a $0$--dimensional subscheme of $\bP^N$.  So $\psi$ is a finite map.  Using the Leray spectral sequence one may infer that $L \to X$ is ample.
\end{proof}

%------------------------------
\subsubsection{Proof of Theorem \ref{th:b}} \label{S:bbtmod}
%------------------------------
Assuming $\olS$ is a compact analytic variety, we want to show that $\Le \to \olS$ is ample.  By Lemma \ref{L:ample} it suffices to show that $\Le \to \hat\wp{}^0$ is both semi-ample and strictly nef. Corollary \ref{C:snef}  establishes that $\Le$ is stricly nef on $\hat\wp{}^0$. It remains to establish the semi-ampleness of $\Lambda_e$ by producing enough sections. The general principle is that the positivity properties of the Hodge bundle lead to big line bundles, and thus many sections. We combine this with an inductive argument to obtain the conclusion.

%------------------------------
\subsubsection*{Step 1: reduced}
%------------------------------

We may assume that $\hat\wp{}^0$ is reduced, cf.~ Step 1 of \cite[pp.~26]{BBT18}.

%------------------------------
\subsubsection*{Step 2: big}
%------------------------------

Assuming $\hat\wp{}^0$ is reduced, the line bundle $\Le \to \hat\wp{}^0$ is big.  This is a consequence of positivity of Hodge bundles; a similar argument is given in  \cite[Lemma 6.4]{BBT18}. Take a desingularization $\pi : \tilde\wp{}^0 \to \olS$.  Then the first Chern class of the pullback $\pi^*\Le \to \tilde\wp{}^0$ is represented by a pseudoeffective $(1,1)$--current by \cite[Theorem 5.20]{Ko2}.  The line bundle $\pi^*(c_1(\Le))$ is big if and only if 
\[
   0 \ < \  
   \int_{\tilde\wp{}^0} \pi^*(c_1(\Le))^{\hat d} \,,
   \quad \hat d = \tdim\,\hat\wp \,,
\]
by \cite[Theorem 1.2]{Bouck}, and this inequality may be inferred from \S\ref{S:chern}.

%------------------------------
\subsubsection*{Step 3: semi-ample}
%------------------------------

To show that $\Le \to \olS$ is semi-ample we have to produce sections.  As in the proof of \cite[Theorem 6.2]{BBT18}, this is done by induction on dimension.  For this we need to produce a $Y \subset |m\Le|$.  By the inductive step $\Le \to Y$ is semi-ample.  Fujita vanishing yields a short-exact sequence
\[
  H^0( \olS , m\Le ) \ \to \ H^0(Y,m\Le) \ \to \ 0 \,.
\]
Thus sections of $m\Le \to \olS$ span the fibres of $\Le\to Y$.  Since $Y = \{ s=0 \}$ for some $s \in H^0(\olS,m\Le)$, we have $s(x) \not=0$ for any $x \not\in Y$.

%------------------------------
\section{Asymptotic behavior of Chern forms} \label{S:IV}
%------------------------------

In this section we give a proof of Theorem \ref{th:c} by establishing the limit \eqref{E:cw} (of which \eqref{E:thc} is a special case).

%------------------------------
\subsection{Formulation of the result}  \label{S:formulation}
%------------------------------

Associated to the bundles $\tGr^p\cF_I \to Z_I^*$ are principal bundles with fiber group 
\[
  \cH^p_I \ = \ \prod_{q=0}^p \tAut(H^{p-q,0}_I(-q)) \ \times \ 
  \prod_{q=p+1}^n \tAut( H^{0,q-p}_I(-p))
  \ = \ 
  \prod_{q=0}^n \tGL( h^{p,q}_I , \bC) \,.
\]
(This discussion includes $\tGr^p\cF \to B$ with $I = \emptyset$.)  Let
\[
  \fh^p_I \ = \ \bigoplus_{q=0}^p \tEnd(H^{p-q,0}_I(-q)) \ \op \ 
  \bigoplus_{q=p+1}^n \tEnd(H^{0,q-p}_I(-p))
  \ = \ 
  \bigoplus_{q=0}^n \mathrm{Mat}(h^{p,q}_I , \bC)
\]
denote the Lie algebra of $\cH^p_I$; here $\mathrm{Mat}(h,\bC)$ is the algebra of $h \times h$ complex matrices.  Let $\bC[\fh^p_I]$ be the algebra of $\bC$--polynomials on $\fh^p_I$, and let 
\[ 
  \mathfrak{P}^p_I \ = \ \bC[\fh^p_I]^{\cH^p_I} \ := \ 
  \left\{
  P \in \bC[\fh^p_I] \ : \ P(X) = P(\tAd_g(X)) \ \forall \ X \in \fh^p_I \,,\ g \in \cH^p_I
  \right\}
\]
be the subalgebra of polynomials that are invariant under the induced action of $\tAd(\cH^p_I)$.  A polarization of the variation of Hodge structure $\Phi_I^0 : Z_I^* \to \Gamma_I\bsl D_I^0$ induces a Hodge metric $h^p_I$ on $\tGr^p\cF_I \to Z_I^*$.  Let $\Theta^p_I$ denote the associated curvature form.  We have the Chern--Weil homomorphism $\mathfrak{P}^p_I \to H^\sb(Z_I^*,\bC)$ mapping $P \mapsto [P(\Theta^p_I)]$.

The extended Hodge bundles $\cFe^p \to \overline B$ induce 
\begin{equation}\label{E:cwinc}
  \cH^p_J \ \inj \ \cH^p_I \quad \hbox{ for all } \quad I \subset J \,,
\end{equation}
(cf.~\S\ref{S:Icoord}) so that $\fP_I^p \subset \fP_J^p$ for all $I \subset J$.  Keeping in mind that $B = Z_\emptyset^*$, we write $\tGr^p\cF = \tGr^p_\emptyset \cF$, $\fP^p = \fP^p_\emptyset$ and $\Theta^p = \Theta^p_\emptyset$.  With this notation we have
\[
  \fP^p \ \subset \ \fP^p_I \quad \hbox{ for all } I \,.
\]  

Fix $P \in \fP^p$.  Without loss of generality $P$ is a homogeneous polynomial of degree $r$ so that $P(\Theta^p) \in \cA^{r,r}(B)$.  Reordering the indices if necessary, we assume that $I = \{1,\ldots,k\}$.  Every point $b \in Z_I^*$ admits a local coordinate chart $(t,w): \olU \to \Delta^k \times \Delta^\ell$ so that $Z_i \cap \olU = \{ t_i=0\}$.  In particular, $Z_I^* \cap \olU = \{ t_1,\ldots,t_k = 0 \}$.  Let 
\[
  \left.P(\Theta^p)\right|_{\sU} \ = \ 
  \sum_{\mystack{|A|+|C|=r}{|B|+|D|=r}} f^p_{\ell,A,B,C,D}\,
  \td w_A \wedge \td \bar w_B \wedge \td t_C \wedge \td \bar t_D
\]
be the local coordinate expression over $\sU = B \cap \olU$, and let
\begin{equation}\label{E:rho}
  \varrho \ = \ 
  \sum_{|A|, |B|=r} f^p_{\ell,A,B}\,
  \td w_A \wedge \td \bar w_B
\end{equation}
be the part of this expression involving only the $\td w$ and $\td \bar w$'s.  Equivalently, the interior products $i_{\partial/\partial t_i} \varrho$ and $i_{\partial/\partial\bar t_i} \varrho$ vanish, and
\begin{equation}\label{E:rhomod}
   \varrho \ \equiv \ \left.P(\Theta^p)\right|_{\sU}
   \qquad \hbox{ mod } \ \td t_i \,,\ \td \bar t_i \,.
\end{equation}
While $\left.P(\Theta^p)\right|_{\sU}$ does not extend to a smooth form on $\olU$, the form $\varrho$ does.  The purpose of this section is to prove

\begin{equation}\label{E:cw} 
  \lim_{t_1,\ldots,t_k\to0} \varrho 
  \ = \ \left.P(\Theta^p_I)\right|_{Z_I^* \cap \olU} \,.
\end{equation}

%------------------------------
\subsubsection*{Outline of the proof of \eqref{E:cw}}  %\label{S:cwoutline}
%------------------------------

This will be a consequence of the analysis of asymptotics utilized by Cattani--Kaplan--Schmid to establish their estimates for the Hodge metric \cite[\S5]{CKS1}.  Given a nondegenerate Hermitian matrix $h = (h_{ab}(t,w))$, define $h^{ab}(t,w)$ by $h^{ac} h_{bc} = \d^a_b$.  The associated curvature matrix is given by $\Theta[h] = (\Theta[h]^a_b)$ with
\[
  \Theta[h]^a_b \ := \ \bar\partial 
  \left( h^{ac} \cdot \partial h_{bc} \right)\,.
\]
The hypothesis $P \in \fP^p$ enters as follows:  Let 
\begin{subequations}\label{SE:rhoh}
\begin{equation}
  \varrho[h] \ = \ (\varrho[h]^a_b)
\end{equation}
be the obvious analog
\begin{equation}
  \varrho[h]^a_b \ := \ - \sum_{i,j} \partial_{\overline w_j} 
  \left( h^{ac} \cdot \partial_{w_i} h_{bc} \right) \, \td w_i \wedge \td \overline w_j
\end{equation}
\end{subequations}
of \eqref{E:rho} for the curvature form $\Theta[h]$.  Then if $A(t) = (A^a_b(t))$ and $B(t) = (B^a_b(t))$ are invertible matrices, and
\begin{subequations}\label{SE:P}
\begin{equation}
  \tilde h_{ab}(t,w) \ = \ A^c_a(t)\,h_{cd}(t,w)\,\overline{B^d_b(t)} \,,
\end{equation}
we have $\varrho[\tilde h]^a_b = (A^{-1})^a_c \,\varrho[h]^c_d \,A^d_b$, so that 
\begin{equation}
  P \big( \varrho[h] \big) \ = \ 
  P \big( \varrho[\tilde h] \big)
  \quad \hbox{for all} \quad P \in \fP^p
\end{equation}
\end{subequations}
by definition of $\fP^p$.  This is important because the metric $h$ blows-up as we approach the divisor $Z$ at infinity.  The several-variable $\tSL(2)$--orbit theorem provides a method for replacing $h$ with an $\tilde h$ that is bounded at infinity (\S\ref{S:lim1}).  One may then argue, via repeated applications of \eqref{SE:P}, that \eqref{E:cw} holds (\S\ref{S:lim2}).  More precisely, we fix a sequence $t_\mu = (t_{\mu 1} , \ldots , t_{\mu k}) \in \Delta^{*k}$ converging to $0$.  Writing $\ell(t_{\mu j}) = z_{\mu j} = x_{\mu j} + \bi y_{\mu j}$, we have $y_{\mu j} = -\tfrac{1}{\pi} \log |t_{\mu j}|$.  Restricting to a subsequence if necessary (and dropping the subscript ${}_\mu$), we may assume without loss of generality that either $y_i/y_j \to 0$, $y_i /y_j \to \infty$, or $y_i/y_j$ is bounded (away from both $0$ and $\infty$).  Reordering indices if necessary, we may assume that there exists $K = \{ k_1 , \ldots , k_\rho \} \subset \{1,\ldots , k\} = I$ so that $y_{k_\a}/y_{k_{\a+1}} \to \infty$ and $y_j/y_{k_\a}$ is bounded for all $k_{\a-1} < j < k_\a$.  In \S\ref{S:ckssl2} we review a collection of commuting $\tSL(2)$'s that is well-suited to studying the asymptotic behavior of $\Theta[h^p](t,w)$ under such a sequence.  The key tool here is a semisimple automorphism $\e(t,w)$ associated with those $\tSL(2)$'s, cf.~\S\ref{S:esp}, and it is the asymptotic behavior of $\tAd_{\e(t,w)}$, and its eigenvalues, that yields the bounded $\tilde h$, cf.~Lemma \ref{L:key1}.

The remainder of \S\ref{S:IV} is occupied with the proof of \eqref{E:cw}.  After recalling the local coordinate expressions for the metrics on the Hodge bundles in \S\ref{S:IVprelim}, \eqref{E:cw} is proved in \S\ref{prf:cw}.

%------------------------------
\subsection{Review of the Hodge bundles and their curvature}   \label{S:IVprelim}
%------------------------------

We begin by reviewing the bundles $\tGr^p_I\cF$ and their curvature forms $\Theta^p_I$, making use of the notations in \S\S\ref{S:Dsplit}--\ref{S:horiz}.

%------------------------------
\subsubsection{The metric on the quotient Hodge vector bundle $\tGr^p\cF$}
%------------------------------

Fix a set $\{ \sfv_a \} \subset \tilde F_0^p$ of linearly independent vectors with the property that $\{ \sfv_a \ \hbox{mod} \ \tilde F_0^{p+1} \}$ forms a basis of $\tGr^p\tilde F_0 = \tilde F^p_0/\tilde F^{p+1}_0$.  With the notations of \S\ref{S:Dsplit}, we set
\[
  \eta(t) \ := \ 
  \exp\left(\bi \,\d_0 \ + \ \textstyle\sum_{i=1}^k \ell(t_i) N_i \right) \,.
\]
Then 
\[
  v_a(t,w) \ := \ \eta(t) \zeta(t,w) \,\sfv_a
\]
is naturally identified with a local holomorphic framing of the quotient Hodge vector bundle $\tGr^p\cF \to \sU$.  Let $\{ v^a(t,w)\}$ denote the dual coframing.  Then the Hermitian metric 
\begin{subequations}\label{SE:hFn}
\begin{equation}
  h(t,w) \ = \ h_{ab}(t,w) \, v^a \ot \bar v^b
\end{equation}
on $\tGr^p\cF$ is given by
\begin{equation} 
  h_{ab}(t,w) \ := \  (\bi)^n Q\left( v_a(t,w) \,,\,
  \overline{ v_b(t,w) }\right) \,.
\end{equation}
\end{subequations}
Define $h^{ab}(t,w)$ by 
\[
  h^{ab}(t,w)\,h_{ac}(t,w) \ := \ \d^b_c \,.
\]
The curvature form on the quotient Hodge vector bundle $\tGr^p\cF \to \sU$ is
\begin{equation}\nonumber 
  \Theta^p \ = \ 
  \bar\partial \left( h^{ac} \cdot \partial h_{bc} \right) \,v_a \ot v^b\,.
\end{equation}

%------------------------------
\subsubsection{The metric the quotient Hodge vector bundle $\tGr^p_I\cF$}\label{S:Icoord}
%------------------------------

Let $X^{p,q}_0$ denote the component of $X$ taking value in $\tilde\fg^{p,q}_0$ (cf.~\S\ref{S:Dsplit}).  Define $X_I(w) = \op\,X^{-p,p}_0(0,w)$, and set 
\begin{equation}\label{E:zetaI}
  \zeta_I(w) \ := \
  \exp\,X_I(w) \,.
\end{equation}
The function $X_I(w)$ is the component of $X(0,w)$  taking value in $\tilde \fz_\ell(0)$ (cf.~\S\ref{S:horiz}).  In particular, both $X_I(w)$ and $\zeta_I(w)$ commute with the $\{ N_j\}_{j=1}^k$.

Recall that $\tilde F^p_0 = \op_{r\ge p} \tilde I^{r,\sb}_0$, so that $\tGr^p\tilde F_0 \simeq \tilde I^{p,\sb}_0 = \op_q \tilde I^{p,q}$.  Refine the set $\{\sfv_a\}$ so that $\sfv_a \in \tilde I^{p,q}_0$ for some $q = q(a)$.  (It is this refined basis that gives us the inclusion \eqref{E:cwinc}.)  Then 
\[
  v^I_a(w) \ := \ \zeta_I(w)\,\sfv_a
\]
defines a local holomorphic framing of the quotient Hodge vector bundle $\tGr^p_I\cF \to \Delta^\ell$.  Let $v^a_I(w)$ denote the dual coframing.  Set 
\begin{equation}\label{E:chN}
  N \ := \ N_1 + \cdots + N_k \tand 
  W \ := \ W(N) \,.
\end{equation}
Then the Hermitian metric
\begin{subequations}\label{SE:hFnI}
\begin{equation}
  h_I(w) \ = \ h^I_{ab}(w) \, v_I^a \ot v_I^b
\end{equation} 
on $\tGr^p_I\cF$ is given by
\begin{equation}
  h^I_{ab}(w) \ = \ 
  \left\{\begin{array}{ll}
    (\bi)^{n-q} \, Q\left( \zeta_I(w) \sfv_a \,,\, 
    N^q \,\overline{\zeta_I(w) \sfv_a} \right) \,,\quad & 
  q = q(a) = q(b) \,,\\
  0 \quad & q(a) \not= q(b)\,.
  \end{array}\right.
\end{equation}
\end{subequations}
Defining $h^{ab}_I(w)$ by 
\[
  h^{ab}_I(w) \, h^I_{ac}(w) \ = \ \d^b_c \,,
\]
the curvature form of $\tGr^p_I\cF \to \Delta^\ell$ is
\begin{equation}\nonumber%\label{E:cFnI}
  \Theta^p_I \ := \ 
  \bar\partial \left( h^{ac}_I \cdot \partial h^I_{bc} \right) \, v^I_a \ot v^b_I\,.
\end{equation}

\begin{remark} \label{R:dfnhI}
In the definitions above, we might just as well replace $N$ with any other $N' = \sum_1^k y^i N_i$ in the nilpotent cone, $y^i > 0$.  Then $W(N) = W(N')$, and the two nilpotent operators are related by $N' = \tAd_g N$ with $g \in \tAut(V_\bR,Q)$ preserving the Deligne splitting of $V_\bC$ determined by the mixed Hodge structure $(W,\tilde F_0)$; that is, $g$ is an element of the group $G_\bR^{\mathrm{DS}_w}$ of \eqref{E:eDS}, \cite[Corollary 4.9]{MR3505643}.  So, while $N$ and $N'$ determine different Hermitian metrics $h_I$ and curvatures $\Theta^p_I$, it follows from the invariance of the Weil homomorphism, that $P(\Theta^p_I)$ is independent of this choice.  (In particular, they determine the same Chern form $c_r(\tGr^p\cF_I)$.)  Of course, this is also a posteriori a consequence of \eqref{E:cw}. 
\end{remark}

%------------------------------
\subsection{Proof of \eqref{E:cw}} \label{prf:cw}
%------------------------------

First we show that the limit on the left-hand side of \eqref{E:cw} exists (Lemma \ref{L:lim}), and then we show that equality holds.  The necessary asymptotics are reviewed in \S\ref{S:CKSasy}.

%------------------------------
\subsubsection{Step 1: the limit exists}\label{S:lim1}
%------------------------------

From \eqref{E:eDS} we see that there exists an invertible matrix $A(u) = (A^b_a(u))$, depending real-analytically on $u$ so that $A^b_a(u) = 0$ if $q(a)\not=q(b)$, and $\sum_b A^b_a(u) \sfv_b$ is an eigenvector of $\bY^\a(u,0)$ with eigenvalue $e_{\a a}$, for all $\a$.  Then, as noted in \eqref{E:Ye}, $\e(s,u;0)$ acts on $A^b_a(u)\sfv_b$ by the eigenvalue
\[
  e_a(s) \ := \ \prod_\a s_\a^{e_{\a a}/2} \,.
\]  
Let $A(u)^{-1} = (B^b_a(u))$ denote the inverse matrix.  Then 
\begin{subequations}\label{SE:ev}
\begin{equation}
  \e(s,u;0) \, \sfv_a \ = \ \sum_{b,c}  B^b_a(u) e_b(s) A_b^c(u) \,\sfv_c \,.
\end{equation}
Notice that $E(s,u) = (E^c_a(s,u))$,
\begin{equation}
  E^c_a(s,u) \ := \ \sum_{b}  B^b_a(u) e_b(s) A_b^c(u)\,,
\end{equation}
\end{subequations}
is the matrix representing $\e(s,u;0)$ with respect to the basis $\{\sfv_a\}$, and we have 
\begin{equation}\label{E:detE}
  \tdet\, E(s,u) \ = \ \prod_c e_c(s) \,.
\end{equation}
Set
\begin{equation}\label{E:tva}
  \tilde \sfv_a(w) \ := \ \psi(w)^{-1}\,\sfv_a \,.
\end{equation}
Then \eqref{E:sl2w} implies $\e(s,u;w)$ acts on $A^b_a(u) \tilde\sfv_b(w)$ by the eigenvalue $e_a(s)$, and 
\[
  \e(s,u;w) \, \tilde\sfv_a(w) \ = \ 
  \sum_c  E^c_a(s,u) \tilde\sfv_c(w)\,.
\]

Define $h^1(u,w) = (h^1_{ab}(u,w))$ by
\[
  h^1_{ab}(u,w) \ = \ 
  Q \left( \exp( 2\bi\,\bN(u) ) \zeta_u(w)\,\sfv_a \,,\, 
  \overline{\zeta_u(w) \,\sfv_b} \right) \,.
\]
Let $\varrho[h^1]$ be defined as in \eqref{SE:rhoh}, substituting $h$ with $h^1$.  Note that \eqref{E:rhomod} yields
\begin{equation}\label{E:Prho}
  \varrho \ = \ P(\varrho[h]) \,.
\end{equation}
With the notation of Remark \ref{R:unifconvg}, we have

\begin{lemma} \label{L:lim}
Suppose that $t \in (\Delta^{\ast k})^K_c$.  Then 
\[
  P\big(\varrho[h](t,w) \big) \ \cksto \ 
  P\big(\varrho[h^1](u,w) \big) \,.
\]
\end{lemma}

\begin{proof}
We have
\begin{eqnarray*}
  v_a(t,w) & = & \eta(t) \zeta(t,w)\,\sfv_a \\
  & = & \ \psi(w) \exp(\textstyle \sum x_j N_j ) 
  \e(s,u;w)^{-1} \lambda(t,w) \e(s,u;w) \, \tilde\sfv_a(w) \\
  & = & \sum_c E^c_a(s,u)\,\psi(w) \exp(\textstyle \sum x_j N_j ) 
  \e(s,u;w)^{-1} \lambda(t,w) \, \tilde\sfv_c(w) \,.
\end{eqnarray*}
Define $\tilde v_c(t,w) := \psi(w) \exp(\textstyle \sum x_j N_j ) 
  \e(s,u;w)^{-1} \lambda(t,w) \, \tilde\sfv_c(w)$,
so that $v_a(t,w) =  E^c_a(s,u)\,\tilde v_c(t,w)$.  Then
\begin{equation} \label{E:hth}
  \renewcommand{\arraystretch}{1.3}
  \begin{array}{rcl}
  h_{ab}(t,w) & = & (\bi)^n \, E^c_a(s,u) \overline{E^d_b(s,u)}
  \, Q \left( \tilde v_c(t,w) \,,\, \overline{\tilde v_d(t,w)} \right)\\
  & = & (\bi)^n \, E^c_a(s,u) \overline{E^d_b(s,u)} \, \tilde h_{cd}(t,w) \,,
\end{array}
\end{equation}
where
\[
  \tilde h_{cd}(t,w) \ := \ 
  Q \left( \tilde v_c(t,w) \,,\, \overline{\tilde v_d(t,w)} \right) \ = \ 
  Q \left( \lambda(t,w)\,\tilde\sfv_c(w) \,,\, 
  \overline{\lambda(t,w)\,\tilde\sfv_d(w)} \right) \,.
\]
Then \eqref{SE:P} yields 
\begin{equation}\label{E:O=tO}
  P(\varrho[h]) \ = \ P(\varrho[\tilde h]) \,.
\end{equation}
From \eqref{E:tva} and Lemma \ref{L:cks2} we see that 
\begin{equation}\label{E:thh1}
  \tilde h(t,w) \ \cksto \ h^1(u,w) \,.  
\end{equation}
Therefore $P(\varrho[\tilde h]) \cksto P(\varrho[h^1])$.  The lemma now follows from \eqref{E:O=tO}. 
\end{proof}

%------------------------------
\subsubsection{Interlude: the Chern form of the Hodge line bundle} \label{S:int}
%------------------------------

We pause in the proof of \eqref{E:cw} to recover a result of Koll\'ar's (Proposition \ref{P:kollar}).  Consider the case that $p = n$, so that $\tGr^p\cF = \cF^n$ is the Hodge vector bundle.  Then $\tdet\,h(t,w)$ is the metric on $\tdet\,\cF^n \to B$.  The asymptotic relationship of the Chern form
\[
  c_1(\tdet\,\cF^n) \ = \ 
  \frac{\bi}{2\pi}\, \overline\partial\partial \log\tdet\,h(t,w)
\]
to the Poincar\'e metric is given by \eqref{E:PMrel}.   From \eqref{E:detE}, \eqref{E:hth} and \eqref{E:thh1} we see that 
\[
  \tdet\,h(t,w) \ = \ \bi^m \,\tdet\,\tilde h(t,w)\,(\tdet\,E(s,u))^2
  \ = \ \,\tdet\,\tilde h(t,w) \prod_a e_a(s)^2 \,,
\]
with $m = n\,\mathrm{rank}\,F^n$, and 
\[
  \tdet\,\tilde h(t,w) \ \cksto \ \tdet\,h^1(u,w) \,.
\]
So (dropping the $\bi^m$)
\begin{equation} \label{E:PMrel0}
\renewcommand{\arraystretch}{1.5}
\begin{array}{rcl}
  \log\tdet\,h(t,w) & = & \displaystyle \log \tdet\,\tilde h(t,w)
  \ + \ \sum_{\a,a} e_{\a a} \log s_\a  \\
  & = & \displaystyle \log \tdet\,\tilde h(t,w) \ + \ 
  \sum_{\a,a} e_{\a a} \log\left(\frac{\log |t_{k_\a}|}{\log | t_{k_{\a+1}}|}\right)\,,
\end{array}
\end{equation}
and 
\begin{equation}\label{E:PMrel1}
  \log \tdet\,\tilde h(t,w) \ \cksto \ \log\tdet\,h^1(u,w)\,.
\end{equation}
Differentiating yields
\begin{eqnarray}
  \nonumber
  \partial_{t_{k_\b}} \log \tdet\,h(t,w) & \cksto & 
  \sum_a \left( e_{\b,a} - e_{\b-1,a} \right) 
  \frac{\td t_{k_\b}}{t_{k_\b} \log |t_{k_\b}|^2} \\
  \label{E:PMrel}
  \partial_{\overline t_{k_\b}}\partial_{t_{k_\b}} \log \tdet\,h(t,w) & \cksto & 
  \sum_a \left( e_{\b,a} - e_{\b-1,a} \right) 
  \frac{\td t_{k_\b} \wedge \td \overline t_{k_\b}}{|t_{k_\b}|^2 \left(\log |t_{k_\b}|^2\right)^2} \,.
\end{eqnarray}

\begin{remark}\label{R:poscoef}
We claim that the coefficients $e_{\b,a} - e_{\b-1,a}$ are all \emph{nonnegative integers}.  The way to see this is to recall that $e_{\b,a}$ is an eigenvector of $\bY^\b$, cf.~\S\S\ref{S:ckssl2}--\ref{S:esp}.  So $e_{\b,a} - e_{\b-1,a}$ is an eigenvector of $\bY^\b - \bY^{\b-1} = \hat\sY_\b$.  The $\sfv_a \in H^{n-q,0}_I(-q)$ are all highest weight vectors for the $\tSL(2)$'s; that is, $\hat\sN_\a^+(u,0) \sfv_a = 0$.  The claim follows from standard $\tSL(2)$--representation theory.
\end{remark}

\begin{proposition}[Koll\'ar {\cite{Ko2}}] \label{P:kollar}
The integral $0 \le \displaystyle \int_B c_1(\tdet\,\cF^n)^{d} < \infty$, $d = \tdim\,B$.
\end{proposition}

\begin{proof}
It suffices to prove
\begin{equation}\label{E:w1}
  \displaystyle \int_\sU c_1(\tdet\,\cF^n)^{k+\ell} < \infty\,.
\end{equation}
Fix $c>1$.  As noted in Lemma \ref{L:cks1}, a neighborhood of $0$ in $\sU$ is covered by (a \emph{finite} number of) sets of the form 
\[
  \sU^K_c \ :=  \ \left\{ (t,w) \in (\Delta^*)^K_c \times \Delta^\ell \ | \ 
  1/c \le u^j_\a \le c \,,\ s_\a > c\right\} \,.
\]
Consequently, we see that to prove \eqref{E:w1} it suffices to show that
\begin{equation}\label{E:w2}
  \int_{\sU^K_c} c_1(\tdet\,\cF^n)^{k+\ell} \ < \ \infty \,.
\end{equation}
From \eqref{E:PMrel0} we see that $c_1(\tdet\,\cF^n) = \frac{\bi}{2\pi}(\eta + \tau)$,  where $\eta = \overline\partial \partial\log \tdet\,\tilde h(t,w)$ and 
\[
  \tau \ := \ 
  \overline\partial \partial 
  \sum_{\a,a} e_{\a a} 
  \log\left(\frac{\log |t_{k_\a}|}{\log | t_{k_{\a+1}}|}\right) 
  \ = \ \sum_a \left( e_{\b,a} - e_{\b-1,a} \right) 
  \frac{\td t_{k_\b} \wedge \td \overline t_{k_\b}}
  {|t_{k_\b}|^2 \left(\log |t_{k_\b}|^2\right)^2} \,.
\]
So to prove \eqref{E:w2}, it suffices to show that 
\begin{equation}\label{E:w3}
  \int_{\sU^K_c} \eta^a \wedge \tau^b \ < \ \infty \,,
\end{equation}
for every $0 \le a,b \in \bZ$ such that $a+b = k+\ell$.  We will prove \eqref{E:w3} by induction on $|K| = \rho$.  

Set $\eta_1 := \overline\partial \partial\log \tdet\,h^1(u,w)$.  Then $\int_{\sU^K_c} \eta_1^a \wedge \tau^b < \infty$.  It follows from \eqref{E:PMrel1} that there exists $c'$ (depending on $c > 1$) so that 
\begin{equation}\label{E:w5}
  \int_{\sV^K_{c,c'}} \eta^a \wedge \tau^b \ < \ \infty \,,
\end{equation}
where 
\[
  \sV^K_{c,c'} \ = \ \left\{ (t,w) \in \sU^K_c \ | \ s_\a > c' \,,\ \forall \ \a \right\} \,.
\]
Notice that 
\[
  \sU^K_c \backslash \ \sV^K_{c,c'} \ = \
  \left\{(t,w) \in \sU^K_c \ | \ \exists \ \a \ \hbox{s.t.}\ c < s_\a \le c' \right\} \,.
\]
In particular, 
\begin{equation}\label{E:w6}
  \sU^K_c \backslash \ \sV^K_{c,c'} \ \subset \ \bigcup_{K' \subsetneq K} \sU^{K'}_{c'} \,.
\end{equation}

If $|K|=1$, then $s = (s_1,\ldots,s_\rho) = (s_1)$.  So that 
\[
  \sU^K_c \backslash \ \sV^K_{c,c'} \ = \
  \left\{(t,w) \in \sU^K_c \ | \ c < s_1 \le c' \right\}
\]
has compact closure in $\sU$.  The desired \eqref{E:w3} then follows from \eqref{E:w5}.

For $|K| > 1$, the desired \eqref{E:w3} now follows from \eqref{E:w6} by induction.
\end{proof}

\begin{remark} \label{R:kollar}
Proposition \ref{P:kollar} also holds when $c_1(\tdet\,\cF^n)$ is replaced with the Chern form $c_1(\Lambda)$ of the augmented Hodge line bundle,
\[
   0 \ \le \ 
   \int_B c_1(\Lambda)^{d} 
   \ < \ \infty \,,
\]
essentially by the same argument.  There is one subtlety here regarding Remark \ref{R:poscoef}: the eigenvalues $e_{\b,a} - e_{\b-1,a}$ appearing in this case need not be non-negative.  However, any negative eigenvalues are dominated by positive eigenvalues.   More precisely, if we define $f_p \in \mathbb{N}$ by expressing the augmented Hodge line bundle \eqref{E:ahlb} as 
\[
  \Lambda \ = \ \bigotimes_p (\cF^p/\cF^{p+1})^{f_p} \,,
\]
then the right-hand side of \eqref{E:PMrel} becomes
\[
  \sum_{p=n}^{\lceil (n+1)/2 \rceil}
  \sum_{\sfv_a \in I^{p,\sb}_0}
  f_p \left( e_{\b,a} - e_{\b-1,a} \right) 
  \frac{\td t_{k_\b} \wedge \td \overline t_{k_\b}}{|t_{k_\b}|^2 \left(\log |t_{k_\b}|^2\right)^2} \,.
\]
Standard $\tSL(2)$--theory ensures that the sum
\[
  \sum_{p=n}^{\lceil (n+1)/2 \rceil}
  \sum_{\sfv_a \in I^{p,\sb}_0}
  f_p \left( e_{\b,a} - e_{\b-1,a} \right) 
\]
is a non-negative integer.
\end{remark}

%------------------------------
\subsubsection{Step 2: equality} \label{S:lim2}
%------------------------------

Returning to \eqref{E:Prho} and Lemma \ref{L:lim}, in order to prove \eqref{E:cw} it remains to show that 
\begin{equation}\label{E:cw2}
  P\big( \varrho[h^1](u,w) \big) \ = \ P \big( \Theta[h_I](w) \big) \,.
\end{equation}

\begin{proof}[Proof of \eqref{E:cw2}]

First recall that $\sum y_j N_j$ commutes with $\xi(0,w)$.  Consequently, $\tAd_{\e(s,u;w)} \sum y_j N_j$ commutes with $\mu(s,u;w)$.  It follows from Lemma \ref{L:mu} and \eqref{E:nulim} that $\bN(u,w)$ and $\psi(w)^{-1} \zeta_u(w)\psi(w)$ commute.  Then \eqref{E:bN} implies $\bN(u) = \bN(u,0)$ and $\zeta_u(w)$ commute.  Finally, we note that \eqref{E:Yw} implies that $\zeta_u(w)$ preserves the $\tilde I_{0,\ell} = \op_{p+q=\ell} \tilde I^{p,q}_0$, for all $\ell$.  These observations, along with the fact that $\bN(u)$ polarizes the MHS $(W,\tilde F_0)$, implies
\begin{equation} \nonumber
  h^1_{ab}(u,w) \ := \ 
  \left\{
  \begin{array}{ll}
  (2\bi)^q\,Q\left(\zeta_u(w) \bN(u)^q \,\sfv_a \,,\, 
  \overline{\zeta_u(w) \,\sfv_b} \right) \,,\ & q = q(a) = q(b) \,;\\
  0 \,, & \hbox{otherwise.}
  \end{array}
  \right.
\end{equation}
The observation \eqref{SE:P} implies that the Hermitian matrix $h^2(u,w) = (h^2_{ab}(u,w))$ given by 
\begin{equation} \nonumber
  h^2_{ab}(u,w) \ := \ 
  \left\{
  \begin{array}{ll}
  Q\left(\zeta_u(w) \bN(u)^q \,\sfv_a \,,\, 
  \overline{\zeta_u(w) \,\sfv_b} \right) \,,\ & q = q(a) = q(b) \\
  0 \,, & q(a) \not= q(b)\,,
  \end{array}
  \right.
\end{equation}
satisfies 
\begin{equation}\label{E:h1=h2}
  P \big( \varrho[h^1](u,w) \big) \ = \ P \big( \varrho[h^2](u,w) \big) \,.
\end{equation}

Each of the cones 
\[
  \s \ := \ \tspan_{\bR_{>0}}\{ N_1 , \ldots , N_k \} \tand
  \hat\s_u \ := \ \tspan_{\bR_{>0}}\{ \hat\sN_1(u,0) , \ldots , 
  \hat\sN_\rho(u,0) \}
\]
is contained in an $\tAd(G_\bR^{\mathrm{DS}_0})$--orbit (cf.~\eqref{E:eDS}), \cite[Corollary 4.9]{MR3505643}.  Additionally, \cite[(4.20.vi)]{CKS1} implies they lie in the same orbit.  In particular, there exists $g(u) \in G_\bR^{\mathrm{DS}_0}$  so that $\bN(u) = \tAd_{g(u)} N$.  Therefore $\bN(u)^q \sfv_a = g(u) N^q g(u)^{-1} \sfv_a$.  Since $g(u)$ preserves both $\tilde I^{n,q}$ and $N^q(\tilde I^{n,q}) = \tilde I^{n-q,0}$, there exist functions $g(u)^a_b$, $q = q(a) = q(b)$, so that 
\begin{equation}\label{E:NvbN}
  \bN(u)^q \sfv_a \ = \ g(u) N^q g(u)^{-1} \sfv_a
  \ = \ g(u)^b_a \, N^q \sfv_b \,.
\end{equation}  
So we have
\begin{equation}\nonumber
  h^2_{ab}(u,w) \ := \ 
  \left\{
  \begin{array}{ll}
  g(u)^c_a\, h^3_{cb}(u,w) \,,\ & 
  q = q(a) = q(b) \,,\\
  0 \,, & q(a) \not= q(b)\,,
  \end{array}
  \right.
\end{equation}
where
\begin{equation}\label{E:h3}
  h^3_{cd}(u,w) \ := \ 
  \left\{
  \begin{array}{ll}
  Q\left(\zeta_u(w) N^q\,\sfv_c \,,\, 
  \overline{\zeta_u(w) \,\sfv_d} \right) \,,\ & q = q(a) = q(b) \,,\\
  0 \,, & q(a) \not= q(b) \,.
  \end{array}
  \right.
\end{equation}
So \eqref{SE:P} yields
\begin{equation}\label{E:h2=h3}
  P \big( \varrho[h^2](u,w) \big) \ = \ P \big( \varrho[h^3](u,w) \big) \,.
\end{equation}

Recall that $\zeta_I(w) = \exp(X_I(w))$, with $X_I(w)$ taking value in
\[
  \fz \ \subset \ \mathfrak{w} \ = \ 
  \bigoplus_{e_1,\ldots,e_\rho\ge0} \fg_{-e_1,\ldots,-e_\rho}(u,0) \,.
\]
(cf.~\eqref{E:zetaI}, \S\ref{S:horiz}, \eqref{E:w} and \eqref{E:zinw}). In fact, \eqref{E:Yw} implies that $X_I(w)$ is the component of $X(0,w) \in \fz$ taking value in 
\[
  \fz \ \cap \ \tker\,\bY^\rho_0 \ = \ 
  \fz \ \cap \ 
  \bigoplus_{e_1,\ldots,e_{\rho-1}\ge0} \fg_{-e_1,\ldots,-e_{\rho-1},0}(u,0)\,,
\]
and that this intersection is independent of $u$.  Consequently, $\zeta_u(w) = \exp X_u(w)$, with $X_u(w)$ the component of $X_I(w)$ taking value in $\fg_{0,\ldots,0}(u,0)$.  It follows from Lemma \ref{L:key1} that 
\begin{equation}\label{E:hpp4}
  \e(s,u;0) \zeta_I(w) \e(s,u;0)^{-1} \ = \ \exp\tAd_{\e(s,u;0)} X_I(w)
  \ \cksto \ \exp X_u(w) \ = \ \zeta_u(w) \,.
\end{equation}
Since $Q$ is $G$--invariant, and $\e(s,u;0)$ is $G_\bR$--valued, we have 
\begin{equation}\label{E:hpp5}
\renewcommand{\arraystretch}{1.3}
\begin{array}{r}
 Q\left(\e(s,u;0)\zeta_I(w) \e(s,u;0)^{-1} N^q\,\sfv_a \,,\, 
  \overline{\e(s,u;0)\zeta_I(w) \e(s,u;0)^{-1}\,\sfv_b} \right)
 \hsp{50pt} \\ = \ 
 Q\left(\zeta_I(w) \e(s,u;0)^{-1} N^q\,\sfv_a \,,\, 
  \overline{\zeta_I(w) \e(s,u;0)^{-1}\,\sfv_b} \right) \,.
\end{array}
\end{equation}
Setting
\[
  h'_{ab}(s,u;w) \ := \ \left\{ \begin{array}{ll}
  Q\left(\zeta_I(w) \e(s,u;0)^{-1} N^q\,\sfv_a \,,\, 
  \overline{\zeta_I(w) \e(s,u;0)^{-1}\,\sfv_b} \right) \,,\ & 
  q = q(a) = q(b) \\
  0 & q(a) \not= q(b)\,,
  \end{array}\right.
\]
\eqref{SE:P}, \eqref{E:h3}, \eqref{E:hpp4} and \eqref{E:hpp5} yield 
\begin{equation}
  \label{E:hpp1}
   P \big(\varrho[h'](s,u;w) \big) \ \cksto \ 
   P \big( \varrho[h^3](u,w) \big) \,.
\end{equation}

On the other hand, by \eqref{SE:ev}, we have $\e(s,u;0)^{-1}\,\sfv_a = E^{-1}(s,u)_a^b\,\sfv_b$.  An analogous argument yields $\e(s,u;0)^{-1}\, N^q \sfv_a = D^{-1}(s,u)^b_a \, N^q \sfv_b$, for some invertible matrix $D(s,u)$.  On the other hand, as noted after \eqref{E:ees}, $\e(s,u;0)$ preserves both $\tilde I^{n,q}_0 = \tspan\{ \sfv_a\}_{q = q(a)}$ and $\tilde I^{n-q,0}_0 = \tspan\{ N^q \sfv_a\}_{q = q(a)}$.  In particular, there exist invertible matrices $A(s,u)^c_a$ and $B(s,u)^d_b$ so that $\e(s,u;0)^{-1} N^q\,\sfv_c = A(s,u)^c_a\, N^q\,\sfv_a$ and $\e(s,u;0)^{-1}\,\sfv_b = B(s,u)^d_b \,\sfv_d$.  Consequently, 
\begin{equation}\label{E:hpp2}
\renewcommand{\arraystretch}{1.7}
\begin{array}{rcl}
  h'_{ab}(s,u;w)
  & = & 
  \displaystyle\sum_{q = q(c)=q(d)}
  D^{-1}(s,u)^c_a \,Q\left(\zeta_I(w)  N^q\,\sfv_c \,,\, 
  \overline{\zeta_I(w)\,\sfv_d} \right) \, \overline{E^{-1}(s,u)^d_b} \\
  & = & 
  \displaystyle\sum_{q = q(c)=q(d)}
  D^{-1}(s,u)^c_a \,h''_{cd}(w) \, \overline{E^{-1}(s,u)^d_b} \,,
\end{array}
\end{equation}
where the Hermitian metric $h''(w) = (h''_{ab}(w))$ is defined by 
\[
  h''_{ab}(w) \ := \ 
  \left\{\begin{array}{ll}
    Q\left( 
    N^q\zeta_I(w) \,\sfv_a \,,\, 
    \overline{\zeta_I(w) \,\sfv_b} \right) \,,\quad & 
  \hbox{if } q = q(a) = q(b) \,,\\
  0 \quad & q(a) \not= q(b)\,.
  \end{array}\right.
\]
So \eqref{SE:P} and \eqref{SE:hFnI} yield
\begin{equation}\label{E:hpp3}
  %P(\Theta[h_I]) \ = \ 
  P(\varrho[h']) \ = \ P(\varrho[h_I]) \,.
\end{equation}
The desired \eqref{E:cw2} now follows from \eqref{E:h1=h2}, \eqref{E:h2=h3}, \eqref{E:hpp1} and \eqref{E:hpp3}.
\end{proof}

\appendix
%------------------------------
\section{Review of asymptotics at infinity} \label{S:rvw}
%------------------------------

The proof of Theorem \ref{th:c} which is given in Section \ref{S:IV}  is based on the properties of the asymptotic behavior of $\Phi$ at infinity established by \cite{Sc, CKS1}. This section is a brief review of the necessary material.

\begin{remark}
It is important to note that  while the results reviewed here were formulated originally in the case  that $D$ is a period domain, they in fact hold in  the more general setting that $D$ is Mumford--Tate domain,  so that there is no obstruction to our applying them here. We refer the reader to  \cite[\S4]{KPR} and the references within.  
\end{remark}

%------------------------------
\subsection{Local VHS}  \label{S:locVHS}
%------------------------------

Recollect that $D$ parameterizes weight $n$, $Q$--polarized Hodge structures on a rational vector space $V$.  

Let 
\[
  \Delta \ := \ \{ \z \in \bC \ : \ |\z| < 1 \}
\]
denote the unit disc, and 
\[
  \Delta^* \ := \ \{ \z \in \bC \ : \ 0 < |\z| < 1 \}
\] 
the punctured unit disc.  Fix a point $b_0 \in Z_I^*$.  Let $\overline \sU \ \simeq \Delta^r = \Delta^k \times \Delta^\ell \ni (t,w)$ be a local coordinate neighborhood centered at $b_0$ in $\overline B$ so that $Z \cap \olU = \{ t_1\cdots t_k = 0 \}$; in particular, 
\[
  \sU \ := \ \overline\sU \cap B \simeq (\Delta^*)^k \times \Delta^\ell \,,
\]
with $r=k+\ell$ and $k = |I|$.  

Let $T_i \in \tAut(V,Q)$ be the local monodromy operator induced by the generator of $\pi_1(\sU)$ coming from the $i$--th $\Delta^*$.  After base change we may (and do) assume the $T_i$ are unipotent, and let $N_i := \log(T_i) \in \tEnd(V,Q)$ denote the nilpotent logarithms. %Then $N_I = \sum_{i\in I}\, N_i$ is the nilpotent monodromy operator about $Z_I^*$. 

Let $\sH \subset \bC$ denote the upper-half plane, and let 
\[
  \tPhi_\sU : \sH^k \times \Delta^\ell \ \to \ D
\]
be a lift of $\left.\Phi\right|_\sU$.  Fix coordinates $(z,w) \in \sH^k \times \Delta^\ell$.  Then $(z,w) \mapsto (\exp(2\pi\sqrt{-1} z) , w)$ defines the covering map $\sH^k \times \Delta^\ell \to (\Delta^*)^k \times \Delta^\ell$.  Here we are writing $\exp(2\pi\sqrt{-1}\,z)$ as short-hand for the $(\Delta^*)^k$--valued $(\exp(2\pi\sqrt{-1}\,z_1) , \ldots , \exp(2\pi\sqrt{-1}\,z_k))$.  Let $\check D \supset D$ denote the compact dual of the Mumford--Tate domain.  Schmid \cite{\Schmid} showed that there exists a holomorphic map 
\[
  F : \olU \ \to \ \check D
\]
so that the lifted period map factors as
\begin{equation}\label{E:tPhiU}
  \widetilde\Phi_\sU(z,w) \ = \ 
  \exp\left( \textstyle \sum_{i \in I} z_i N_i \right) 
  \cdot  F(\exp (2\pi\bi z),w) \,.
\end{equation}

Let
\[
  \ell(t_j) \ := \ \frac{\log t_j}{2\pi \bi} \,.
\]
Observe that 
\begin{subequations}\label{SE:locVHS}
\begin{equation}
  \Phi_\sU(t,w) \ := 
  \exp\left( \textstyle \sum_{i \in I} \ell(t_i) N_i \right) 
  \cdot F(t,w)
\end{equation}
defines a \emph{local variation of Hodge structure}
\begin{equation}
  \Phi_\sU : \sU = (\Delta^*)^k \times \Delta^\ell \ \to \ 
  \Gamma_\sU\backslash D \,,
\end{equation}
\end{subequations}
where $\Gamma_\sU \subset \Gamma$ is the \emph{local monodromy group} generated by the unipotent monodromy operators $\{T_i\}_{i\in I}$.  Notice that \eqref{SE:locVHS} recovers $\left.\Phi\right|_\sU$ after quotienting $\Gamma_\sU\backslash D \to \GsD$ by the full monodromy group $\Gamma$.   

%------------------------------
\subsection{The period map $\Phi_I^0 : Z_I^* \to \Gamma_I\bsl D_I^0$} \label{S:PhiI}
%------------------------------

Define 
\[
  \s_I \ = \ \tspan_{\bR_{>0}}\{ N_i \ | \ i \in I \} \,.
\]
and let 
\[
  W_0(\s_I) \ \subset \ W_1(\s_I) \ \subset \cdots \subset \ W_{2n}(\s_I) 
\]
be the associated (shifted\footnote{\label{fn:W}Typically, ``$W(N)$'' denotes a representation-theoretic filtration with indexing that is centered at $0$.  In this paper, we are letting ``$W(N)$'' denote the shifted geometric filtration $W(N)[-n]$, with indexing that centered at $n$.  The reasons for this mild abuse of notation is that (i) this is the only filtration we will work with and (ii) the abuse significantly reduces notational clutter.}) weight filtration.  The filtration $F(0,w)$ induces a Hodge filtration on the ``weight graded'' quotients
\[
  \tGr^{W(\s_I)}_a \ := \ W_a(\s_I)/W_{a-1}(\s_I) \,.
\]
In general, the filtration $F(0,w) \in \check D$ depends on the choice of local coordinates.  (The well-defined object is the nilpotent orbit $\exp(\bC\s_I)\cdot F(0,w)$.)  However, the Hodge filtration on $\tGr^{W(\s_I)}_a$ is independent of this choice.  So we have a well-defined map on $Z_I^* \cap \olU$ sending $(0,w)$ to the Hodge filtrations $F^\tinyb_{(0,w)}(\tGr^{W(\s_I)}_\tinyb)$.  This map takes value in 
\[
  D_I \ = \ \{ F^\tinyb(\tGr^W(\s_I)_\tinyb) \ | \ (W,F,\s_I) \hbox{ is a 
  polarized MHS}\} \,.
\]
The fact that the triples $(W,F,\s_I)$ are \emph{polarized} mixed Hodge structures implies that $D_I$ may be realized as a Mumford--Tate subdomain of a product of period domains.  The local map $Z_I^* \cap \olU \to D_I$ induces the period map
\begin{equation}\label{E:PhiIdfn}
  \Phi^0_I : Z_I^* \ \to \ \Gamma_I\backslash D_I \,.
\end{equation}

%------------------------------
\subsection{Deligne's $\bR$--split PMHS} \label{S:Dsplit}
%------------------------------

In general, the LMHS $(W,F(0,w))$ will not be $\bR$--split.  It will be convenient to work with Deligne's associated $\bR$--split MHS $(W(N),\tilde F_w)$, cf.~\cite[(2.20)]{CKS1}; here $\tilde F_w = \exp(-\bi\,\delta_w) \cdot F(0,w)$.  The element $\delta_w \in \fg_\bR$ commutes with the $N_i$, and $w \mapsto \d_w$ is a real analytic map $\Delta^\ell \to \fg_\bR$.  Let 
\begin{equation}\label{E:DSw}
  V_\bC \ = \ \bigoplus\tilde I^{p,q}_w \tand
  \fg_\bC \ = \ \bigoplus\tilde\fg^{p,q}_w
\end{equation}
be the associated Deligne splittings.  Set
\[
 \tilde \fn_w \ := \ \bigoplus_{\mystack{p<0}{q}} \tilde\fg^{p,q}_w 
 \ = \ \tilde\fg^{-,\sb}_w \,.
\]
Note that 
\begin{equation}\label{E:gwdecomp}
  \fg_\bC \ = \ \tilde \fn_w \,\op\, \tilde \fp_w\,,
\end{equation}
where $\tilde\fp_w$ is the Lie algebra of $\tStab_{G_\bC}(\tilde F_w)$.   Shrinking the neighborhood $\olU$ if necessary, there exists a unique holomorphic map
\[
  X : \overline\sU \ \to \ \tilde\fn_0
\]
so that $X(0,0) = 0$ and 
\begin{subequations}\label{SE:P5a}
\begin{equation}
  \Phi(t,w) \ = \ 
  \exp\left(\bi \,\d_0 \ + \ \textstyle\sum_{i=1}^k \ell(t_i)N_i \right)
   \zeta(t,w) \cdot \tilde F_0 \,, 
\end{equation}
where
\begin{equation}
  \zeta(t,w) \ := \ \exp X(t,w) \,.
\end{equation}
\end{subequations}
See \cite[\S5]{CKS1} for details.

%------------------------------
\subsection{Horizontality}  \label{S:horiz}
%------------------------------
Horizontality implies $X(0,w)$ commutes with the $\{ N_i \ | \ i \in I\}$, and therefore takes value in 
\[
  \fz_I \ := \ 
  \bigcap_{i \in I} \,\tker(\tad\,N_i) 
  \ \subset \ \fg_\bC \,.
\]
Without loss of generality, $I = \{1,\ldots,k\}$, and we may write $\fz_I$ as 
\begin{equation}\nonumber%\label{E:fzC}
  \fz_\bC \ := \ 
  \{ Z \in \fg_\bC \ | \ [Z,N_j] = 0 \,,\ \forall j \} 
  \ = \ \bigcap_{j=1}^k \tker\,(\tad\,N_j) \,.
\end{equation}
This Lie algebra inherits a decomposition
\[
  \fz_\bC \ = \ \bigoplus_{\ell\le0} \fz_\ell(w) \,,\quad
  \hbox{with}\quad
  \fz_\ell(w) \ := \ \fz_\bC \,\cap\, \op_{p+q=\ell}\,\tilde\fg^{p,q}_w
\]  
from the Deligne splitting.

%------------------------------
\subsection{Relationship between the $\bR$--split $\tilde F_w$}\label{S:reltFw}
%------------------------------

The centralizer
\[
  \cZ_\bC \ := \ \{ g \in G_\bC \ | \ \tAd_g N_i = N_i \,,\ \forall i \} \,.
\]
is defined over $\mathbb{Q}$, preserves the weight filtration $W(N)$, $N = N_1 + \cdots + N_k$, and has Lie algebra $\fz_\bC$. 

The Hodge filtrations $\tilde F_w$ are all congruent to $\tilde F_0$ under the action of $\cZ_\bR$ \cite{KP16}.  In particular, we may choose a real analytic function $\psi : \Delta^\ell \to \exp(\fz_\bR) \subset \cZ_\bR$  so that $\psi(0)$ is the identity and 
\begin{equation} \label{E:Fw}
  \psi(w) \cdot \tilde F_w \ = \ \tilde F_0 \,.\footnote{This choice of $\psi(w)$ is not unique.  There is a unique choice of $\psi_\bC : \Delta^\ell \to \exp(\fz_\bC \cap \tilde\fn_0) \subset \cZ_\bC$ so that $\psi_\bC(0)$ is the identity and $\psi_\bC(w) \cdot \tilde F_w = \tilde F_0$.  Nonetheless it is better to work with the $\cZ_\bR$--valued $\psi(w)$, because the Hermitian metric $h_{F^n}$ is $G_\bR$--invariant, but not $G_\bC$--invariant.  And ultimately the argument and result are independent of our choice.}
\end{equation}

%------------------------------
\subsection{The Cattani--Kaplan--Schmid asymptotics} \label{S:CKSasy}
%------------------------------

Here we review the necessary results from \cite[\S5]{CKS1}.\footnote{The arguments of \cite[\S5]{CKS1} assume that $\ell = 0$, so that the holomorphic parameter $w$ does not play a role.  However, the proofs there (up to and including that of \cite[(5.14)]{CKS1}) all apply, in a straightforward manner, in our more general setting to yield the assertions below.}  

%------------------------------
\subsubsection{The CKS coordinates} \label{S:CKScoords}
%------------------------------

Fix $K = \{ k_1 , \ldots , k_\rho \} \subset \{1,\ldots , k\}$ so that $1 \le k_1 < \cdots < k_\rho = k$.  Define
\begin{eqnarray*}
  s_\a & := & y_{k_\a}/y_{k_{\a+1}} \,,\quad \hbox{for } \a < \rho \,, 
  \tand s_\rho \ := \ y_k \,; \\
  u_\a^j & := & y_j/ y_{k_\a} \,,\quad \hbox{for } \ k_{\a-1} < j < k_\a \,.
\end{eqnarray*}
Define
\begin{eqnarray*}
  \bR_+^k & := & \{ y = (y_j) \in \bR^k \ | \ y_j > 0 \} \\
  \bR_+^\rho & := & \{ s = (s_\a) \in \bR^\rho \ | \ s_\a > 0 \} \\
  \bR_+^{k-\rho} & := & \{ u = (u^j_\a) \in \bR^{k-\rho} \ | \ u^j_\a > 0 \} \,.
\end{eqnarray*}
Let
\begin{eqnarray*}
  \sA & := & \hbox{(real) analytic functions of } (u,w) \in 
  \bR^{k-\rho}_{+} \times \Delta^\ell\,,\\
  \sL & := & \hbox{Laurent polys.~in } \{s_\a^{1/2}\} \hbox{ with coef.~in } \sA \,,\\
  \sO & := & \hbox{pullback to $\sH^k \times \Delta^\ell$ of the 
  ring of holo.~germs at } 
  0 \in \Delta^r = \Delta^k \times \Delta^\ell \\
  & & \hbox{via } \sH^k \to (\Delta^*)^k \inj \Delta^k \,,\\
  \sL^\flat & := & \hbox{polys.~in } s_\a^{-1/2} \hbox{ with coef.~in } \sA \,,\\
  (\sO \ot \sL)^\flat & := & \hbox{subring of $\sO \otimes \sL$ gen.~by 
  $\sO$, $\sL^\flat$, and all monomials of the form}\\
  & & t_j s^{m_1/2}_1 \cdots s^{m_\rho/2}_\rho 
  \hbox{ with $m_\a \in \bZ$ and $m_\a \not=0$ only if $j \le i_\a$.}
\end{eqnarray*}
We identify $\bR^k_{+}$ with $\bR^{\rho}_{+} \times \bR^{k-\rho}_{+}$ by $y \mapsto (s,u)$.  Recall that $\sH \subset \bC$ denotes the upper-half plane.  Given $c > 0$ define
\begin{eqnarray*}
  (\bR^k_+)^K_c & := & \left\{ y \in \bR^k_+ \ | \ s_\a > c \,,\ 
  1/c \le u^j_\a \le c \right\}  \\
  (\sH^k_+)^K_c & := & \left\{ z \in \sH^k \ | \ z = x + \bi y \,,\ 
  y \in (\bR^k_+)^K_c \right\}\\
  (\Delta^{\ast k})^K_c & := & \left\{ t \in (\Delta^*)^k \ | \ 
  \ell(t_j) \in (\sH^k)^K_c \right\} \,.
\end{eqnarray*}

\begin{lemma}[{\cite[(5.7)]{CKS1}}] \label{L:cks1}
\emph{(a)} For any $c > 1$, the regions $(\Delta^{\ast k})^K_c$ corresponding to the various permutations of the variables and choices of $K \subset \{1,\ldots,k\}$ cover the intersection of $(\Delta^*)^k$ with a neighborhood of $0 \in \Delta^k$.
\emph{(b)} The set $(\sO \ot \sL)^\flat$ consists of precisely those elements in $\sO \ot \sL$ that are bounded on $(\sH^k)^K_c$ for some (any) $c$.
\end{lemma}

\begin{remark}\label{R:unifconvg}
The coordinates $y = (s,u)$ are well-adapted to study the asymptotic behavior of the Hodge metric and Chern form for the sequence $t_\mu$ (cf.~the outline of the proof of \eqref{E:cw} in \S\ref{S:formulation}).  For the remainder of \S\ref{S:CKSasy}, and throughout \S\ref{prf:cw}, the notation
\[
  f(x,s,u;w) \ \cksto \ g(x,u;w) \quad \hbox{(or } f(t,w) \ \cksto \ g(x,u;w)\hbox{)}
\]
will indicate that \emph{$f(x,s,u;w)$ converges to $g(x,u;w)$ as $s_1 , \ldots , s_\rho \to \infty$, and that this convergence is uniform on compact subsets of $\{ x \in \bR^k\} \times \{ u^j_\a \in \bR^{k-\rho}_+\} \times \{ w \in \Delta^\ell\}$.}
\end{remark}

%------------------------------
\subsubsection{Commuting $\tSL(2)$'s} \label{S:ckssl2}
%------------------------------

Define
\[
  \sN_\a(u) \ := \ 
  N_{k_\a} \ + \ \sum_{k_{\a-1} < j < k_\a} u^j_\a N_j 
  \ = \ \frac{1}{y_{k_\a}} \sum_{k_{\a-1}<j\le k_\a} y_j N_j\,.
\]
Note that
\begin{equation}\label{E:NvsN}
  \sum_{j=1}^k y_j N_j \ = \ 
  \sum_{\a=1}^\rho (s_\a s_{\a+1} \cdots s_\rho) \sN_\a(u) \,.
\end{equation}
Since each $\exp(\sum z_j N_j) \cdot \tilde F_w$ is a nilpotent orbit, it follows that $\exp( \sum_\a z_{k_\a} \sN_\a(u) ) \cdot \tilde F_w$ is also a nilpotent orbit.  The several-variable $\tSL(2)$--orbit theorem \cite[(4.20)]{CKS1} associates to this nilpotent orbit a collection $\{ \nu_\a : \tSL(2,\bC) \to G_\bC\}_{\a=1}^\rho$ of commuting horizontal $\tSL(2)$'s.  Let $\{ \hat\sN_\a(u,w) \,,\, \hat\sY_\a(u,w) \,,\, \hat\sN_\a^+(u,w)\}_{\a=1}^\rho$ denote the $\nu_\a$--images of the standard generators of $\fsl(2,\bR)$.  Each of $\hat \sN_\a$, $\hat\sY_\a$ and $\hat\sN_\a^+$ is a $\fg_\bR$--valued member of $\sA$.   Furthermore, 
\begin{equation}\label{E:sl2w}
  \left\{ \hat\sN_\a(u,w) \,,\, \hat\sY_\a(u,w) \,,\, \hat\sN_\a^+(u,w)
  \right\}
  \ = \ \tAd_{\psi(w)}^{-1} \left\{ 
  \hat\sN_\a(u,0) \,,\, \hat\sY_\a(u,0) \,,\, \hat\sN_\a^+(u,0)
  \right\} \,,
\end{equation}
for all $1 \le \a \le \rho$.

\begin{proof}[Proof of \eqref{E:sl2w}]
This is a consequence of \eqref{E:Fw} and \cite[(4.75)]{CKS1}: note that the functions $T = T(W,F)$ and $\Phi = \Phi(Y,F)$ of \cite[p.~506]{CKS1} are $G_\bR$--equivariant.
\end{proof}

Define
\begin{equation}\label{E:bYa}
  \mathbf{Y}^\a(u,w) \ := \ 
  \sum_{\b=1}^\a \hat\sY_\b(u,w) \ \stackrel{\eqref{E:sl2w}}{=} \ 
  \tAd_{\psi(w)}^{-1} \bY^\a(u,0) \,.
\end{equation}
It follows directly from the CKS--construction that 
\begin{equation} \label{E:Yw}
  \bY^\rho(u,w) \quad \hbox{is the element of $\fg_\bR$ acting on 
  $\tilde \fg_{w,\ell}$ by $\ell \in \bZ$;}
\end{equation}
in particular, $\bY^\rho_w := \bY^\rho(u,w)$ is independent of $u$.  

%------------------------------
\subsubsection{Eigenspace decompositions}\label{S:esp}
%------------------------------

Set
\begin{equation}\label{E:dfne}
  \e(y,w) \ = \ \e(s,u;w) \ := \ 
  \exp\left( \half \sum_\a \log s_\a \bY^\a(u,w) \right) 
  \ = \ \tAd_{\psi(w)}^{-1} \e(s,u;0)\,.
\end{equation}
Recall that the eigenvalues of $\bY^\a(u,w)$ are integers.  Since $\bY^\a$ depends real-analytically on $(u,w)$, both the eigenvalues and their multiplicities are independent of $(u,w)$, and the eigenspaces depend real-analytically on $(u,w)$.
\begin{equation}\label{E:Ye}
\begin{array}{l}
  \hbox{If $\bY^\a(u,w)$ acts by the eigenvalue $e_\a \in \bZ$,}\\
  \hbox{then $\e(s,u;w)$ acts by the eigenvalue $\prod_\a s_\a^{e_\a/2}$.}
\end{array}
\end{equation}
So $\e(s,u;w)$ is a $G_\bR$--valued function in $\sL$.  Additionally $\bY^\a(u,w) \in \tilde\fg^{0,0}_{w,\bR}$, so that $\bY^\a(u,w)$ preserves the Deligne splittings \eqref{E:DSw}.  Consequently, 
\begin{equation}\label{E:eDS}
  \e(s,u;w) \ \in \ 
  G_\bR^{\mathrm{DS}_w} \ := \ 
  \{ g \in G_\bR \ | \ g( \tilde I^{p,q}_w ) = 
  \tilde I^{p,q}_w \,,\ \forall \ p,q \} \,.
\end{equation}
This has two consequences: first, 
\begin{equation}\label{E:eFw}
  \e(s,u;w) \cdot \tilde F_w \ = \ \tilde F_w \ = \ \e(s,u;w)^{-1} 
  \cdot \tilde F_w \,.
\end{equation}
Second, since $\e(s,u;w)$ is semisimple, 
\begin{equation}\label{E:ees}
  \begin{array}{l}
  \hbox{$\tilde I^{p,q}_w$ decomposes into a direct sum of 
  $\e(s,u;w)$--eigenspaces.}
  \end{array}
\end{equation}

Since the $\{\bY_\a(u,w)\}_{\a=1}^\rho$ are commuting semisimple endomorphisms, the Lie algebra admits a simultaneous eigenspace decomposition
\[
  \fg_\bR \ = \ \bigoplus_{e_1,\ldots,e_\rho \in\bZ} 
  \fg_{e_1,\ldots,e_\rho}(u,w) \,,
\]
with $\tad\,\bY^\a$ acting on $\fg_{e_1,\ldots,e_\rho}$ by the eigenvalue $e_\a \in \bZ$.  In particular,
\begin{equation}\label{E:e}
  \tAd_{\e(s,u;w)} \hbox{ acts on } \fg_{e_1,\ldots,e_\rho}(u,w) 
  \hbox{ by the eigenvalue } \prod_\a s_\a^{e_\a/2} \,.
\end{equation}
The eigenspaces $\fg_{e_1,\ldots,e_\rho}(u,w)$ depend real-analytically on $(u,w)$, and \eqref{E:bYa} implies 
\[
  \fg_{e_1,\ldots,e_\rho}(u,w) \ = \ 
  \tAd_{\psi(w)}^{-1} \fg_{e_1,\ldots,e_\rho}(u,0) \,.
\]

Recollect that the common intersection of the weight filtrations
\begin{equation}\label{E:w}
  \mathfrak{w} \ := \ 
  \bigcap_{\a=1}^\rho W_0 \left( \tad(\sN_1 + \cdots + \sN_\a) \right) \ = \ 
  \bigoplus_{e_1,\ldots,e_\rho \ge 0} \fg_{-e_1,\ldots,-e_\rho}(u,w)
\end{equation}
is the direct sum of the eigenspaces for the nonpositive eigenvalues.  The following is the  \emph{key lemma} in our analysis of the asymptotic behavior of the curvature matrix.

\begin{lemma}\label{L:key1}
Suppose that $U(u,w) \in \mathfrak{w}$ depends continuously on $(u,w)$.  Let $U'(u,w)$ denote the component of $U(u,w)$ taking value in 
\[
  \fg_{0,\ldots,0}(u,w) \ = \ \{ X \in \fg \ | \ 
  [\bY^\a(u,w),X] = 0 \,,\ \forall \ \a\}
  \ = \ \bigcap_\a \tker(\tad\,\bY^\a(u,w))\,,
\]
with respect to the decomposition \eqref{E:w}.  Then \eqref{E:e} yields
\begin{equation} \nonumber %\label{E:lim2}
  \e(s,u;w) \,\exp(U(u,w))\,\e(s,u;w)^{-1}
  \ \cksto \ \exp(U'(u,w)) \,,
\end{equation}
\emph{cf.~Remark \ref{R:unifconvg}}.
\end{lemma}

\begin{proof}
This is an immediate consequence of \eqref{E:e}.  
\end{proof}

%------------------------------
\subsubsection{Asymptotic behavior}
%------------------------------

We will find it useful to write
\begin{subequations}\label{SE:P5b}
\begin{equation}
  \eta(t)\,\zeta(t,w) \ = \ \psi(w) \exp(\textstyle \sum x_j N_j )\nu_w(y) \xi(t,w) \psi(w)^{-1} \,,
\end{equation}
(cf.~\S\ref{S:Dsplit}), where
\begin{equation}
\renewcommand{\arraystretch}{1.3}
\begin{array}{rcl}
  \nu_w(y) & := & 
  \exp \bi \left(\tAd_{\psi(w)}^{-1} \d_0 \ + \ 
  \textstyle\sum_{i=1}^k y_i N_i \right) \\
  \xi(t,w) & := & \psi(w)^{-1} \z(t,w) \psi(w) 
  \ = \ \exp \tAd_{\psi(w)}^{-1} X(t,w) \,.
\end{array}
\end{equation}
This allows us to rewrite \eqref{SE:P5a} as
\begin{equation}
  \Phi(t,w) \ = \ \psi(w) \exp(\textstyle \sum x_j N_j )\nu_w(y) \xi(t,w) 
  \cdot \tilde F_w \,.
\end{equation}
\end{subequations}
Define
\[
  \mu(s,u;w) \ := \ \e(s,u;w) \xi(0,w) \e(s,u;w)^{-1} \,.
\]

Recall that $X(0,w) = \log \zeta(0,w)$ takes value in the centralizer $\fz = \cap_j\,\tker(\tad N_j)$ of the $\{N_j\}_{j=1}^k$ (\S\ref{S:horiz}).  Notice that 
\begin{equation}\label{E:zinw}
  \fz \ \subset \ \bigcap_{\a=1}^\rho
  \tker \left(\tad(\sN_1 + \cdots + \sN_\a) \right) \ \subset \ 
  \bigcap_{\a=1}^\rho W_0 \left( \tad(\sN_1 + \cdots + \sN_\a) \right) 
  \ = \ \mathfrak{w}\,.
\end{equation}
Let $X_u(w)$ denote the component of $X(0,w)$ taking value in $\fg_{0,\ldots,0}(u,0)$ with respect to the decomposition $\fg_\bC = \op\,\fg_{e_1,\ldots,e_\rho}(u,0)$, and set $\zeta_u(w) := \exp(X_u(w))$.  Then \eqref{E:bYa} implies $U'(u,w) = \tAd_{\psi(w)}^{-1} X_u(w)$.

\begin{lemma} \label{L:mu}
Both $\mu$ and $\mu^{-1}$ belong to $\sL^\flat$, and 
\[
  \mu(s,u;w) \ \cksto \ 
  \exp\tAd_{\psi(w)}^{-1} X_u(w) \ = \ \psi(w)^{-1} \z_u(w) \psi(w) \,.
\]
\end{lemma}

\begin{proof}
Write $\xi(0,w) = \exp(U(w))$ with $U(w) = \tAd_{\psi(w)}^{-1} X(0,w)$.  The result now follows from Lemma \ref{L:key1}
\end{proof}

Define $\lambda = \lambda_1 \cdot \lambda_2$ by 
\begin{eqnarray*}
  \lambda_1(t,w) & := & \e(s,u;w) \, \nu_w(y) \, \e(s,u;w)^{-1} \\
  \lambda_2(t,w) & := & \e(s,u;w) \, \xi(t,w) \, \e(s,u;w)^{-1} \,,
\end{eqnarray*}
so that 
\begin{equation}\label{E:Phi-lam}
\renewcommand{\arraystretch}{1.5}
\begin{array}{rcl}
  \eta(t)\,\zeta(t,w) & = & 
  \psi(w) \exp\left(\textstyle\sum x_i N_i\right) 
  \e(t,w)^{-1} \lambda(t,w) \e(t,w) \psi(w)^{-1} \,,\\
  \Phi(t,w) & = & \psi(w) \exp\left(\textstyle\sum x_i N_i\right) 
  \e(t,w)^{-1} \lambda(t,w) \e(t,w) \cdot \tilde F_w \,.
\end{array}
\end{equation}

Set
\begin{equation}\label{E:bN}
  \bN(u,w) \ := \ \sum_\a \hat\sN_\a(u,w) 
  \ \stackrel{\eqref{E:sl2w}}{=} \ \tAd_{\psi(w)}^{-1} \bN(u,0)  \,.
\end{equation}
To simplify notation we set
\[
  \bN(u) \ := \ \bN(u,0) \,.
\]

\begin{lemma}[{\cite[\S5]{CKS1}}] \label{L:cks2}
Both $\lambda$ and $\lambda^{-1}$ belong to $(\sO \otimes \sL)^\flat$, and 
\[
  \lambda(t,w) \ \cksto \ \psi(w)^{-1} \exp\left(\bi\,\bN(u,0) \right) 
  \zeta_u(w) \psi(w) \,.
\]
\end{lemma}

\noindent When $\ell=0$, the lemma is proved by Cattani--Kaplan--Schmid in \cite[pp.~511--512]{CKS1}.  Their argument extends to the general case with only minor modification; we sketch the proof here for completeness.

\begin{proof}
The proof of \cite[(5.12)]{CKS1} applies here to yield
\begin{equation}\label{E:nulim}
\renewcommand{\arraystretch}{1.5}
\begin{array}{rcl}
  \tAd_{\e(s,u;w)} \tAd_{\psi(w)}^{-1} \d_0  & \cksto &  0 \,,\\
  \tAd_{\e(s,u;w)} \textstyle\sum_{j=1}^k y_j N_j & \cksto & \bN(u,w) \,,
\end{array}
\end{equation}
so that 
\begin{equation} \label{E:lim1}
  \lambda_1(t,w) \ \cksto \ \exp(\bi\,\bN(u,w)) \,.
\end{equation}
Briefly, \eqref{E:nulim} is a consequence of Lemma \ref{L:key1}, and the facts (a) that
\[
  \tAd_{\psi(w)}^{-1}\d_0 \ \ \in \ \ 
  \fz \ \cap \ \bigoplus_{p,q\le-1} \tilde\fg^{p,q}_w 
  \ \ \subset \ \
  \bigoplus_{\mystack{e_1,\cdots,e_{\rho-1}\ge0}{e_\rho\ge2}}
  \fg_{-e_1,\ldots,-e_\rho}(u,w)
\]
and (b) the observation \eqref{E:NvsN} and the fact \cite[(4.20.iii)]{CKS1} that $\hat\sN_\a(u,w)$ is the component of $\sN_a(u)$ taking value in 
\[
  \fz \ \cap \ \bigcap_{\b=1}^{\a-1}\tker(\tad\,\bY^\b(u,w))
  \ \subset \ 
  \bigoplus_{e_\a,\ldots,e_\rho\ge0} 
  \fg_{0,\ldots,0,-e_\a,\ldots,-e_\rho}(u,w) \,.
\]

From \eqref{SE:P5a} and \eqref{SE:P5b} we see that the nilpotent orbit asymptotically approximating $\Phi(t,w)$ is
\begin{eqnarray*}
  \theta(t,w) & = & 
  \eta(t)\zeta(0,w) \cdot \tilde F_0 \\
  & = & 
  \psi(w) \exp(\textstyle \sum x_j N_j )\nu_w(y) \xi(0,w) 
  \cdot \tilde F_w \\
  & = & 
  \psi(w) \exp(\textstyle \sum x_j N_j )\,
  \e(s,u;w)^{-1} \lambda_1(s,u;w) \e(s,u;w) \,\xi(0,w) \cdot \tilde F_w \,.
\end{eqnarray*}
Fix a $G_\bR$--invariant distance $d$ on $D$.  Keeping \eqref{E:eFw} and \eqref{E:Phi-lam} in mind, the nilpotent orbit theorem \cite[(1.15.iii)]{CKS1} implies 
\[
  d \left( \Phi(t,w) \,,\, \theta(t,w) \right) \ = \ 
  d \left( \lambda(t,w)\cdot\tilde F_w \,,\,
  \lambda_1(t,w)\, \mu(s,u;w) \cdot \tilde F_w \right)
  \ \cksto \ 0 \,.
\]
It then follows from Lemma \ref{L:mu} and \eqref{E:lim1} that
\[
  \lambda(t,w)\cdot\tilde F_w \ \cksto \ 
  \exp(\bi\,\bN(u,w)) \cdot \psi(w)^{-1} \z_u(w) \psi(w) \cdot \tilde F_w \,.
\]
Therefore
\[
  \lambda_2(t,w)\cdot\tilde F_w \ \cksto \ 
  \psi(w)^{-1} \z_u(w) \psi(w) \cdot \tilde F_w \,.
\]
Since both $\lambda_2(t,w)$ and $\psi(w)^{-1} \z_u(w) \psi(w)$ take value in $\exp(\tilde\fn_w)$, it follows from \eqref{E:gwdecomp} that
\begin{equation}\label{E:lim3}
  \lambda_2(t,w) \ \cksto \ \psi(w)^{-1} \z_u(w) \psi(w) \,.
\end{equation}
The lemma now follows from \eqref{E:bN}, \eqref{E:lim1} and \eqref{E:lim3}.
\end{proof}

%------------------------------
\subsection*{Acknowledgements}
%------------------------------
This is a significant revision of our 2017 draft (arXiv:1708.09523v1). While our claims have been scaled back, we remain optimistic on the validity of the main conjecture. We thank several people for feedback on the original version, particularly Wushi Goldring and Patrick Brosnan for some very relevant comments. We are also grateful to the two anonymous referees who had very pertinent comments that helped us significantly improve our manuscript. 

Beyond this revision, several significant developments have happened since our original draft. Specifically, some of the material has been spun off and expanded (see esp. \cite{phvb} and \cite{GGRinfty}). On a different track,  Bakker--Brunebarbe--Tsimerman \cite{BBT18} have established important results that overlap with some of our considerations (by significantly different methods). In the interest of concision and progressing the field, we have decided to freely use the results of \cite{GGRinfty} and \cite{BBT18} here (while \cite{phvb} is somewhat dependent on Section \ref{S:IV} of our paper). 

While the subject of periods and moduli is central in algebraic geometry with a venerable tradition, our interest in the questions addressed here was rekindled by an NSF FRG grant for which two of the authors were PIs (RL/DMS-1361143 and CR/DMS-1361120), and the other two (MG and PG) were frequent collaborators. We acknowledge the NSF support and thank the other PIs (P. Brosnan, M. Kerr, and G. Pearlstein) for many stimulating discussions and sharing with us some related ideas. 
 
%------------------------------
\bibliographystyle{amsalpha}

\providecommand{\bysame}{\leavevmode\hbox to3em{\hrulefill}\thinspace}
\providecommand{\MR}{\relax\ifhmode\unskip\space\fi MR }

\providecommand{\MRhref}[2]{%
  \href{http://www.ams.org/mathscinet-getitem?mr=#1}{#2}
}
\providecommand{\href}[2]{#2}

%------------------------------

%------------------------------
\end{document}